\documentclass[twoside,10pt, a4paper]{article}
\makeatletter
\def\blfootnote{\gdef\@thefnmark{}\@footnotetext}
\makeatother
\usepackage[ansinew]{inputenc}
\usepackage[T1]{fontenc}
\usepackage{amsmath}
\usepackage{amsthm}
\usepackage{amscd, amsfonts, amssymb}
\usepackage[all]{xy}
\usepackage{mathrsfs}
\usepackage[dvips]{graphicx}
\usepackage{epic, eepic}
\usepackage[colorlinks = true, citecolor = purple, urlcolor  = purple, linkcolor = violet]{hyperref}
\usepackage{xcolor}

\usepackage{titling}
\usepackage{titlesec}
\titleformat*{\section}{\normalfont\Large\color{blue!80!black}}
\titleformat*{\subsection}{\normalfont\large\color{blue!80!black}}
\titleformat*{\subsubsection}{\normalfont\normalsize\color{blue!80!black}}
\usepackage[dotinlabels]{titletoc}
\usepackage{abstract}
\usepackage{helvet}         
\usepackage{courier}        
\usepackage{fix-cm}
\usepackage{BOONDOX-calo}

\numberwithin{equation}{section}

\theoremstyle{definition}
\newtheorem{dfn}{Definition}[section]
\newtheorem{example}{Example}[section]

\newtheorem{remark}{Remark}[section]
\theoremstyle{definition}
\newtheorem{prop}{Proposition}[section]
\newtheorem{thm}[prop]{Theorem}

\newtheorem{corollary}[prop]{Corollary}

\newtheorem*{thm*}{Theorem}

\renewcommand{\fam}[1]{\mathsf{#1}}

\newcommand{\Z}{\mathbb{Z}}
\newcommand{\Q}{\mathbb{Q}}
\newcommand{\N}{\mathbb{N}}

\newcommand{\rrangle}{\,\rangle\!\!\!\rangle\,}
\newcommand{\llangle}{\langle\!\!\!\langle\,}
\newcommand{\Com}{\mathsf{Com}}
\newcommand{\Mod}{\mathsf{Mod}}

\newcommand{\azu}{\mathtt{azu}}
\newcommand{\ram}{\mathtt{ram}}
\newcommand{\branch}{\mathtt{branch}}

\newcommand{\Spec}{\mathrm{Spec}}

\newcommand{\Specm}{\mathrm{Specm}}

\newcommand{\Max}{\mathtt{Max}}

\newcommand{\SDer}{\mathcal{D\!e\!r}}

\newcommand{\eps}{\underline{\boldsymbol{e}}}

\newcommand{\oo}{\mathfrak{o}}
\newcommand{\pp}{\mathfrak{p}}

\newcommand{\mm}{\mathfrak{m}}
\newcommand{\Mm}{\mathfrak{M}}
\newcommand{\Pp}{\mathtt{P}}
\newcommand{\ee}{\mathbf{e}}
\newcommand{\cc}{\mathbf{c}}
\renewcommand{\aa}{\mathbf{a}}
\newcommand{\KW}{\mathcal{W}^{1/n}}
\newcommand{\eKW}{\mathbf{W}^{1/n}}

\newcommand{\eJac}{\mathcal{J}}

\newcommand{\W}{\mathcal{W}}
\newcommand{\eW}{W}
\newcommand{\OO}{\mathcal{O}}

\renewcommand{\AA}{\mathcal{A}}

\newcommand{\MM}{\mathcal{M}}

\newcommand{\Env}{\mathbf{E}}
\newcommand{\cent}{\mathfrak{Z}}

\newcommand{\uu}{\boldsymbol{u}}
\newcommand{\vv}{\boldsymbol{v}}

\newcommand{\yy}{\underline{y}}
\newcommand{\ww}{\underline{w}}

\newcommand{\Xscr}{\mathbcal{X}}

\newcommand{\wps}{\underline{\boldsymbol{w}}}

\DeclareMathOperator{\id}{id}

\DeclareMathOperator{\Aut}{Aut}
\DeclareMathOperator{\End}{End}
\DeclareMathOperator{\Hom}{Hom}
\DeclareMathOperator{\Ext}{Ext}

\DeclareMathOperator{\Der}{\mathrm{Der}}
\DeclareMathOperator{\Gal}{Gal}

\usepackage{hyperref}

\begin{document}
\pretitle{\begin{flushleft}\LARGE \scshape
} 
\posttitle{\par\end{flushleft}
\rule[8mm]{\textwidth}{0.1mm}
}
\preauthor{\begin{flushleft}\Large \scshape
\vspace{-5mm}}
\postauthor{\end{flushleft}\vspace{-8mm}}                                                                                                                                                                                                                                                                                                                                                                                                                                                           
\title{Kummer--Witt--Jackson algebras}
\author{\large Daniel Larsson}
\date{}

\maketitle
\vspace{-0.5cm}
\begin{abstract}  
\vspace{0.2cm}
\noindent This paper\blfootnote{\textit{e-mail:}  \href{mailto:daniel.larsson@usn.no}{daniel.larsson@usn.no}} is concerned with the construction of a small, but non-trivial, example of a polynomial identity algebra, which we call the \emph{Jackson algebra}, that will be used in sequels to this paper to study non-commutative arithmetic geometry. In this paper this algebra is studied from a ring-theoretic and geometric viewpoint. Among other things it turns out that this algebra is a ``non-commutative family'' of central simple algebras and thus parametrises Brauer classes over extensions of the base. 
\end{abstract}
\section{Introduction}

The geometric study of polynomial identity (PI) algebras, and in particular maximal orders, has seen a growing interest at least since the early 1990's, in particular by the school following Michael Artin. For a recent example see \cite{ChanNyman}. Most of this study is concerned with algebras over algebraic surfaces (the case of curves is rather well-understood) and there are beautiful results already, but a complete classification, in particular over non-algebraically closed fields (which is my main interest), seems to be out of reach at the moment. 

However, the algebras that will appear in this paper are algebras over higher-dimensional schemes, and the main example, denoted $\eJac_x$, will turn out to live over a three-fold, say $X$ for now. This three-fold actually seems to be rational in many cases. On the other hand, the restriction of $\eJac_x$ to divisors on $X$ will give a particular example of the algebras studied in the literature (see the already mentioned \cite{ChanNyman} and the references therein), at least when extending to the projective closure of the centre $\Spec(\cent(\eJac_x))$. In fact, $X=\Spec(\cent(\eJac_x))$.  

My impetus for studying this algebra is actually two-fold (no pun intended). The first is that this algebra appeared as a $q$-deformation of the Lie algebra $\mathfrak{sl}_2$ a long time ago in a paper I wrote with S.D. Silvestrov \cite{LaSi}. In that paper was left a few open questions regarding some ring-theoretic and homological properties concerning this algebra which we weren't able to solve at that moment. A few years later I was able to do this, but I never really wrote it up in any readable manner and it was shelved. Then a few years ago I became interested in applying non-commutative deformation theory to arithmetic geometry and I needed a simple, but non-trivial, algebra to use as testing ground. I remembered the algebra in quarantine. Whence the second reason for doing a somewhat detailed ring-theoretic analysis of this algebra.

The construction of this algebra $\eJac_x$ will proceed in several steps. First, following the origins of the $q$-deformed $\mathfrak{sl}_2$, we construct a family of hom-Lie algebras (see section \ref{sec:twisted_derivations}) and we also introduce a new class of such, \emph{infinitesimal hom-Lie algebras}, that will play a specific r\^ole in a later paper concerning torsion points on elliptic curves. Then, in section \ref{sec:enveloping}, we construct ``enveloping algebras'' of the hom-Lie algebras and restrict to a specific type of algebras which we call \emph{Kummer--Witt hom-Lie algebras} and their enveloping algebras, \emph{Kummer--Witt algebras}. We do all this globally over a general scheme, which is strictly not necessary for the rest. 

Primarily in order to simplify notation, we therefore restrict to the affine case and study in section \ref{sec:ring}, the ring-theoretic and homological properties of these Kummer--Witt algebras. In particular, the algebra $\eJac_x$ will make its entrance as a canonical subalgebra of such a Kummer--Witt algebra (see section \ref{sec:J}). The section begins with a reminder on the infinitesimal structure of polynomial identity algebras and the definition of ``Auslander-regularity''. 

In section \ref{sec:centreBS} we will use a (unpublished) method of A. Bell and S. P. Smith \cite{BellSmith} to compute the centre and prove, among other things, that after a base extension to the algebraic closure of the base field, the centre has rational singularities (see theorem \ref{thm:rationalCM}). 

Section \ref{sec:fibre} begins by introducing a simplified version of what is to be viewed as a non-commutative scheme for us. A detailed and general definition can be found in \cite{LarssonAritGeoLargeCentre}. The crucial property of a non-commutative space is that it has a more intricate infinitesimal structure captured by the existence of non-trivial $\Ext^1$-groups between different points. Elements in these groups are to be interpreted as ``tangents'' between points. Also, in this section is a definition of what is to be meant by an ``$L$-rational point'' on a non-commutative scheme. Next we define a family of quadratic divisors on the non-commutative space, $\Xscr_{\eJac_x}$, associated $\eJac_x$.

After that, in section \ref{sec:fibres_A}, comes the first arithmetic discussion of $\eJac_x$. Namely, we prove that fibres in $\Xscr_{\eJac_x}$ over an open subscheme $S$ of the central subscheme $X=\Spec(\cent(\eJac_x))$ parametrises classes in Brauer groups by these fibres containing symbol algebras. We also prove that this fibration over $X$ includes quantum Weyl algebras over the complement of $S$ in $X$. 

Finally, in section \ref{sec:rational_points} we look at rational points on $\Xscr_{\eJac_x}$. First we find all one-dimensional points (the ``commutative points'') and compute the tangent structure of these points, i.e., compute the $\Ext^1$-groups. After that we look at higher-dimensional points and use a (slightly modified) construction of D. Jordan \cite{Jordan} to construct families of rational points that, over the algebraic closure, contains all rational points up to isomorphism. There are two types of higher-dimensional points: torsion points and torsion-free points. We prove that there are only a discrete set of torsion points, but a continuous family of torsion-free points. We conclude by giving a somewhat detailed example.  

\subsubsection*{Acknowledgements} 
As already indicated, this paper has travelled a long and winding road. Its primary ancestor was an attempt to answer some ring-theoretic questions raised in \cite{LaSi} back in 2006. However, as the years went by, it evolved into something completely different. Along the way, I have benefited a lot from discussions and help from Ken Brown, Arvid Siqveland, Eli Matzri, Daniel Chan and Fred Van Oystaeyen. 

\subsection*{Notation}
We will adhere to the following notation throughout.
\begin{itemize}
	\item All rings are unital.
	\item For a general algebra (not necessarily commutative) $\Mod(A)$ denotes the category of  left $A$-modules.
	\item The notation $\Max(A)$ denotes the \emph{set} of maximal ideals, while $\Specm(A)$ denotes the maximal spectrum of $A$ (if $A$ is commutative). 
	\item The notation $\Mod(A)$, denotes the set (groupoid) of isoclasses of all $A$-modules. All modules are \emph{left} modules unless otherwise explicitly specified. The class of modules with annihilator ideal being prime is denoted $\Mod^\Delta(A)$. 
	\item $\cent(A)$ denotes the centre of $A$. 
	\item For $\pp$ a prime in $A$, $k(\pp)$ denotes the residue class field of $\pp$. 
	\item Abelian sheaves are denoted with scripted letters.
	\item All schemes and algebras are noetherian. Schemes are also assumed to be separated. Many results surely hold without this assumption but it is cumbersome to keep track of this hypothesis in any given situation so we make the blanket assumption of separatedness throughout for simplicity.
\end{itemize}

\section{Algebras of twisted derivations}\label{sec:twisted_derivations}
Let $X_{/S}$ be an $S$-scheme and let $\AA$ be a coherent sheaf of 
\emph{commutative} $\OO_X$-algebras. Assume that $\sigma$ is an algebra endomorphism on $\AA$. Then a \emph{global ($\sigma$-)twisted $S$-derivation} on $\AA$ is an operator
$$\Delta_U(ab)=\Delta_U(a)b+\sigma(a)\Delta_U(b),\quad a,b\in \AA(U), \quad \text{with $U\subseteq X$ open},$$ 
\begin{example}
The canonical example of a $\sigma$-derivation is a map $\AA\to\AA$ on the form
$$\Delta_\sigma(U):=a_U(\id-\sigma):\,\,\AA(U)\to\AA(U), \quad a_U\in\AA(U), \quad \text{with $U\subseteq X$ open}.$$In fact, for many algebras these types of maps are the only $\sigma$-derivations available. 
\end{example}

Below $U\subseteq X$ will always denote an open subset of the scheme $X$. We define the $\AA$-module $\mathrm{Ann}_\AA(\Delta)$ as 
$$\mathrm{Ann}_\AA(\Delta)(U):=\Big\{a\in\AA(U)\mid a\Delta(b)=0,\,\, \text{for all}\,\, b\in\AA(U)\Big\}.$$Assume that 
$$\Delta_U\circ\sigma=q_U\cdot\sigma\circ\Delta_U, \quad q_U\in\AA(U), \quad \sigma\left(\mathrm{Ann}(\Delta)\right)\subseteq \mathrm{Ann}(\Delta),$$ and form the left $\AA$-module $\AA\cdot\Delta$ by
$$(\AA\cdot\Delta)(U):=\AA(U)\cdot\Delta_U.$$On $\AA\cdot\Delta$ we introduce the product $\llangle\,\cdot,\,\cdot\rrangle$ by
$$\llangle a\cdot\Delta_U,b\cdot\Delta_U\rrangle_U:=\sigma(a)\cdot\Delta_U(b\cdot\Delta_U)-\sigma(b)\cdot\Delta_U(a\cdot\Delta_U),\quad a,b\in\AA(U).$$
We now have the following theorem. 
\begin{thm}\label{thm:twistprod}The above product is $\OO_S$-linear and satisfies 
\begin{itemize}
\item[(i)] $\llangle a\cdot\Delta_U, b\cdot\Delta_U\rrangle_U=
  (\sigma(a)\Delta_U(b)-\sigma(b)\Delta_U(a))
  \cdot\Delta_U$;  
\item[(ii)] $\llangle a\cdot \Delta_U, a\cdot\Delta_U\rrangle_U=0$;
\item[(iii)] $\circlearrowleft_{a,b,c}\Big(\llangle \sigma(a)\cdot
  \Delta_U,\llangle
  b\cdot\Delta_U,c\cdot\Delta_U\rrangle_U\rrangle_U+ q_U\cdot\llangle
  a\cdot\Delta_U,\llangle
  b\cdot\Delta_U,c\cdot\Delta_U\rrangle_U\rrangle_U\Big)=0$,
\end{itemize} where $a,b,c\in\AA(U)$.
\end{thm}
The proof of this global version is simply a standard descent argument using the affine version as given in \cite{HaLaSi}. 

The following definition was introduced in \cite{LarArithom} as a generalisation of a \emph{hom-Lie algebra}. For our purposes the definition as given below is certainly overkill but we introduce it in its full generality nonetheless. Example \ref{exam:mainHom} gives the main example relevant for us. 

Let $G$ denote a finite group scheme acting on $X$ over $S$, and let $\AA$
be an $\OO_X[G]$-sheaf of $\OO_X$-algebras. This means that $\AA$ is an $\OO_X$-algebra together with a $G$-action, compatible with the $G$-action on $X$ in the sense that $\sigma(x a)=\sigma(x)\sigma(a)$, $x\in\OO_X$, $a\in\AA$. 

Let $\AA\langle G\rangle$ denote the skew-group algebra of $G$ over $\AA$. Recall that this is the free algebra $\AA\{g_1, g_2, \dots, g_n\}$, $g_i\in G$, with product defined by the rule
$$(ag_i)\cdot (bg_j)=ag_i(b)g_ig_j.$$This defines an associative algebra structure. 

\begin{dfn}\label{dfn:globhom}Given the above data, a ($G$-)\emph{equivariant hom-Lie
  algebra on $X$ over $\AA$} is an $\AA\langle G\rangle$-module $\MM$ together
  with, for each open $U\subset X$, an $\OO_X$-bilinear product $\llangle\,\cdot,\cdot\,\rrangle_{U}$ on
  $\MM(U)$ such that 
  \begin{description} 
       \item[(hL1.)] $\llangle a,a\rrangle_{U} =0$, for all $a\in \MM(U)$;
       \item[(hL2.)] for all $\sigma\in G$ and for each $\sigma$ a $q_\sigma\in \AA(U)$, the identity $$\circlearrowleft_{a,b,c}\Big \{\llangle \sigma(a),\llangle
       b,c\rrangle_{U}\rrangle_{U}+q_\sigma\cdot\llangle a,\llangle
       b,c\rrangle_{U}\rrangle_{U}\Big\}=0,$$ holds. 
  \end{description} A morphism of equivariant hom-Lie algebras $(\MM, G)$ and
       $(\MM', G')$ is a pair $(f,\psi)$ of a morphism of
       $\OO_X$-modules $f: \MM\to\MM'$ and $\psi: G\to G'$ such 
       that $f\circ \sigma=\psi(\sigma)\circ f$, and $f({U})\big(\llangle
       a,b\rrangle_{\MM; U}\big) =\llangle  
       f({U})(a),f({U})(b)\rrangle_{\MM'; U}$. 
\end{dfn}
Hence, an equivariant hom-Lie algebra is a family of (possibly isomorphic) products
para\-metrised by $G$. 
\begin{dfn}\label{dfn:homlie}A product $\llangle\,\cdot,\cdot\,\rrangle_\sigma$ in the equivariant structure,
for fixed $\sigma\in G$, is a \emph{hom-Lie algebra on $\MM$}.
\end{dfn}
\begin{example}\label{exam:mainHom}The main example for us in this paper is: the $\AA$-module $\AA\cdot\Delta_\sigma$, $\sigma\in G$ defines an equivariant hom-Lie algebra by theorem \ref{thm:twistprod}. Each $\sigma$ gives a hom-Lie algebra structure on $\AA\cdot\Delta_\sigma$.  
\end{example}
Some reasons why $\sigma$-derivations are important (and actually prevalent in abundance) in arithmetic and geometry, can be found in \cite{LarArithom} and the references therein. 
\subsection{Infinitesimal hom-Lie algebras}\label{sec:infhom} Let $X^!$ be an \emph{infinitesimal thickening} of $X$. This means that $X$ is defined as a closed subscheme of $X^!$ by a nilpotent sheaf of ideals $\mathcal{I}$. The \emph{order} of the thickening is defined as the least integer $n$ such that $\mathcal{I}^n=0$. By construction $X$ and $X^!$ have the same underlying topological space. Notice that $\OO_{X^!}$ is an $\OO_X$-algebra. 

Let $\sigma$ be an $\OO_S$-linear automorphism of $X^!$. This induces, by definition, an $\OO_S$-linear automorphism $\sigma_X$ on $X$, i.e., $\sigma_X=\sigma\vert_X$. We will for simplicity assume that $\sigma$ is a lift of the identity on $X$. In other words, $\sigma\vert_X=\id_X$. 

Put, for all open $U\subseteq X$,
$$\Delta_\sigma(U):=a_U(\id-\sigma): \,\, \OO_{X^!}(U)\to \OO_{X^!}(U), \quad a_U\in\OO_{X}(U).$$Then $\big(\OO_{X^!}\cdot\Delta_\sigma, \llangle\,,\,\rrangle\big)$ is called an \emph{infinitesimal hom-Lie algebra on $X$}. 

The canonical example is the following. To simplify the discussion, we restrict to affine schemes. Everything globalises without problem. So, let $X=\Spec(R)$, with $R\in\Com(k)$. Put $\underline{t}:=\{t_1, t_2, \dots, t_d\}$. Then $X^!:=\Spec\big(R[\underline{t}]/(\underline{t})^n\big)$ is an $n$-th order infinitesimal thickening of 
$$X=\Spec(R)=\Spec\Big(\big(R[\underline{t}]/(\underline{t})^n\big)/(\underline{t})\Big).$$We will look at the particular case of $d=1$. Put $R^!:=R[t]/(t)^n$. Then $\eps_i:=t^i\Delta_\sigma$, $0\leq i\leq n-1$, is a basis for $R^!\cdot\Delta_\sigma$ as an $R^!$-module. We will consider the automorphism $\sigma(t)=qt$, with $q\in k^\times$, $q\neq 1$, and $a\in R$. Notice that $\sigma_R=\id$. 

Let $\Delta_\sigma:=a(1-q)^{-1}(\id-\sigma).$ Then a simple induction argument gives that $\Delta_\sigma(t^i)=a[i]_qt^i$, where we have put $[i]_q:=\frac{1-q^n}{1-q}=1+q+q^2+\cdots+q^{i-1}$. Notice that $[0]_q=0$ and $[1]_q=1$. Using Theorem \ref{thm:twistprod} (i) a small computation gives that 
$$\llangle \eps_i,\eps_j\rrangle=a\big(q^i[j]_q-q^j[i]_q\big)\eps_{i+j}.$$Observe that when $i+j\geq n$, then $\llangle \eps_i,\eps_j\rrangle=0$ since in that case $t^{i+j}=0$. In addition, $\llangle\eps_0,\eps_i\rrangle = a[i]_q\eps_i$, for all $i$. 
\begin{example}When $n=2$, we get the solvable $R$-Lie algebra
$$\llangle \eps_0,\eps_1\rrangle = a\eps_1.$$
\end{example}
\begin{example}When $n=3$, we also get a solvable $R$-Lie algebra:
$$\llangle \eps_0,\eps_1\rrangle =a\eps_1,\qquad \llangle \eps_0,\eps_2\rrangle =a[2]_q\eps_2,\qquad \llangle \eps_1,\eps_2\rrangle=0.$$ One would be tempted to conjecture that these are Lie algebras for all $n$. However, this is not true as the case $n=4$ shows.
\end{example}
\begin{example}
	So when $n=4$ we get 
	$$\llangle \eps_0,\eps_i\rrangle =a[i]_q\eps_i,\qquad \llangle \eps_1,\eps_2\rrangle =aq\eps_2,\qquad \llangle \eps_1,\eps_2\rrangle=0, \qquad \llangle\eps_1,\eps_2\rrangle=0.$$ This is not a Lie algebra since, for instance,
	$$\circlearrowleft_{0,1,2}\llangle \eps_0,\llangle \eps_1,\eps_2\rrangle\rrangle =a^2q(q-1)\eps_3,$$ which is not zero unless $a=0$ (trivial) or $q=1$. 
\end{example}
That the case $n=2$ gives a Lie algebra is quite natural, but it seems that this should be the case also for $n=3$, is more of a coincidence. 

\section{Enveloping algebras}\label{sec:enveloping}
We will now use Theorem \ref{thm:twistprod} to construct an ``enveloping'' algebra. By this we mean an associative algebra $\Env$ constructed on the given non-associative structure.  

This algebra $\Env$ is constructed as follows. Let $\AA$ be a finitely generated (commutative) $\OO_{X}$-algebra, $\sigma\in\Aut_{\OO_X}\!(\AA)$ and let $\Delta\in\SDer_{\!\!\sigma}(\AA)$, such that over $U\subseteq X$, $$\Delta_U:=\alpha_U\cdot(\id-\sigma),\quad \alpha_U\in \AA(U).$$

Then Theorem \ref{thm:twistprod} endows $\AA\cdot\Delta$ with a non-associative algebra structure. It is clear that, over $U$, the elements $$\boldsymbol{\eps}^{\underline{k}}:=y_1^{k_1}y_2^{k_2}\cdots y_n^{k_n}\cdot \Delta_U, \qquad \underline{k}\in\mathbb{Z}_{\geq 0}^n,$$ form a basis over $U$ for $\AA\cdot\Delta$ as an $\AA$-module, where $y_1, y_2, \dots, y_n$ are generating sections of $\AA$ over $U$. Then we have the relations
$$(\boldsymbol{\eps}^{\underline{k}})^\sigma\cdot\boldsymbol{\eps}^{\underline{l}}\,\,-\,\,
(\boldsymbol{\eps}^{\underline{l}})^\sigma\cdot \boldsymbol{\eps}^{\underline{k}}=\llangle\boldsymbol{\eps}^{\underline{k}},\boldsymbol{\eps}^{\underline{l}}\rrangle,$$ so we can form
$$\Env(\AA\cdot \Delta)(U):=\frac{\OO_{X}(U)\big\{\boldsymbol{\eps}^{\underline{k}}\mid \underline{k}\in \mathbb{Z}_{\geq 0}^n\big\}}{\Big((\boldsymbol{\eps}^{\underline{k}})^\sigma\cdot\boldsymbol{\eps}^{\underline{l}}\,-\,
(\boldsymbol{\eps}^{\underline{l}})^\sigma\cdot \boldsymbol{\eps}^{\underline{k}}\,-\,\llangle\boldsymbol{\eps}^{\underline{k}},\boldsymbol{\eps}^{\underline{l}}\rrangle\Big)}.$$Obviously this is in general a very complicated algebra because it is infinitely presented, exactly as the universal enveloping algebra of the classical Witt--Lie algebra for instance. Things simplify considerably if $\AA$ is finite as an $\OO_X$-module. 

So assume that $\AA$ is locally free of (constant) finite rank as $\OO_X$-algebra, given over $U$ by
$$\AA(U)=\OO_X(U)y_1\oplus \OO_X(U)y_2\oplus\cdots \oplus \OO_X(U)y_n.$$Put
$$\eps_i:=y_i\cdot \Delta_U$$and
$$\ee_{ij}:=(\eps_i)^\sigma\cdot\eps_j-\,
(\eps_j)^\sigma\cdot\eps_i-\,\llangle\eps_i,\eps_j\rrangle.$$ Then 
\begin{equation}\label{eq:env}
\Env(\AA\cdot \Delta)(U)=\OO_{X}(U)\big\{\eps_1,\eps_2,\dots,\eps_n\big\}\Big/(\ee_{ij}).
\end{equation} 
\subsection{Enveloping algebras of infinitesimal hom-Lie algebras}We continue with the instance $d=1$ as to not get too bogged down in awkward notation. That is, we consider the thickening $R^!=R[t]/(t)^n$. This means that $\AA$ from the previous section, corresponds to $R^!$. 

Here $\eps_i^\sigma=(t^i\Delta_\sigma)^\sigma=\sigma(t^i)\Delta_\sigma=q^it^i\Delta_\sigma$, with $\Delta_\sigma=\id-\sigma$. From this we see
\begin{align*}
	(\eps_i)^\sigma\cdot\eps_j-(\eps_j)^\sigma\cdot\eps_i&= \sigma(t^i)\Delta_\sigma(t^j\Delta_\sigma)-\sigma(t^j)\Delta_\sigma(t^i\Delta_\sigma)\\
	&=q^it^i\Delta_\sigma(t^j\Delta_\sigma)-q^jt^j\Delta_\sigma(t^i\Delta_\sigma)\\
	&=q^i\eps_i\eps_j-q^j\eps_j\eps_i,
\end{align*}whence
\begin{align*}
\ee_{ij}=q^i\eps_i\eps_j-q^j\eps_j\eps_i-a\big(q^i[j]_q-q^j[i]_q\big)\eps_{i+j}.
\end{align*}Therefore,
$$\Env(R^!\cdot\Delta_\sigma)=\frac{R\{\eps_0,\eps_1,\dots,\eps_{n-1}\}}{\Big(\eps_i\eps_j-q^{j-i}\eps_{j}\eps_i-a([j]_q-q^{j-i}[i]_q)\eps_{i+j}\Big)},$$where we have divided by $q^i$ for aesthetic reasons. 

We continue the examples from Section \ref{sec:infhom}.
\begin{example}When $n=2$ we get
$$\Env(R^!\cdot\Delta_\sigma)=\frac{R\{\eps_0,\eps_1\}}{\big(\eps_0\eps_1-q\eps_1\eps_0-a\eps_1\big)}.$$By changing basis $\eps_0\mapsto \eps_0+a(1-q)^{-1}$, we see that this algebra is in fact isomorphic to the famous quantum plane $\mathbf{Q}_{R,q}^2=R[x,y]/(xy-qyx)$. Observe that the algebra $\llangle \eps_0,\eps_1\rrangle = a\eps_1$ is a Lie algebra but $\Env(R^!\cdot\Delta_\sigma)$ is not the universal enveloping algebra for this Lie algebra. In other words, first order thickenings (or deformations) give quantum planes! 
\end{example}
\begin{example}In the case $n=3$ we compute the relations 
$$\eps_0\eps_1-q\eps_1\eps_0=a\eps_1,\quad 	\eps_0\eps_2-q^2\eps_2\eps_0=a[2]_q\eps_2,\quad
	\eps_1\eps_2-q\eps_2\eps_1=0,$$ and so 
$$\Env(R^!\cdot\Delta_\sigma)=\frac{R\{\eps_0,\eps_1,\eps_2\}}{\begin{pmatrix}
 \eps_0\eps_1-q\eps_1\eps_0- a\eps_1\\
  \eps_{0}\eps_2-q^2\eps_2\eps_{0}- a[2]_q\eps_{2}\\
  \eps_{1}\eps_2-q\eps_2\eps_{1}
\end{pmatrix}}.$$Using the same change of basis as in the previous example, we get the isomorphic algebra
\begin{equation}\label{eq:Aq3_inf}
\Env(R^!\cdot\Delta_\sigma)\simeq \mathbf{Q}^3_{R,q}:=\frac{R\{\eps_0,\eps_1,\eps_2\}}{\Big(
 \eps_0\eps_1-q\eps_1\eps_0,\,\,
  \eps_{0}\eps_2-q^2\eps_2\eps_{0},\,\,
    \eps_{1}\eps_2-q\eps_2\eps_{1}
\Big)}.
\end{equation}The ring $\mathbf{Q}^3_{R,q}$ is a quantum affine three-space. This means that $\mathbf{Q}^3_{R,q}$ comes associated with a solvable Lie algebra. 

We will be able to say more concerning these algebras later as they have nice ring-theoretic  properties (see section \ref{sec:ncKummerWitt_spaces}). 
\end{example}
We leave the case $n=4$ for the reader.
\subsection{Kummer--Witt hom-Lie algebras}\label{sec:Kummer_Witt}
 We keep the notation from above and further denote the algebra structure on $\AA$ over $U$ by 
$$y_iy_j=\sum_{k=0}^n a^k_{ij}y_k, \quad a_{ij}^k\in\OO_X(U),$$ where the $y_i$ are the algebra generators of $\AA$ over $U$. Let $\sigma$ be the $\OO_X$-linear algebra morphism on $\AA$ defined by $\sigma(y_i)=q_iy_i$, $q_i\in\OO_X(U)$ and let $\Delta$ be the $\sigma$-derivation from the previous section. Put $\eps_i:=y_i\cdot\Delta$. Then, from Theorem \ref{thm:twistprod}, the pair 
\begin{equation}\label{eq:Witthom}
\W^\sigma_\AA:=(\AA\cdot \Delta, \llangle\,,\,\rrangle), \qquad \llangle \eps_i,\eps_j\rrangle=a\sum_{k=0}^n(q_i-q_j)a^k_{ij}\eps_k
\end{equation} defines a hom-Lie algebra structure on $\AA\cdot\Delta$. We call $\W^\sigma_\AA$ the \emph{Witt hom-Lie algebra} over $X$ attached to $\AA$ and $\sigma$. 
\begin{remark}The construction just given is obviously not dependent on the particular choice $\sigma(y_i)=q_iy_i$. Any other automorphism can be used. However, the result will, of course, be more complicated and harder to write out.
\end{remark}
From now on we assume that the $n$-th roots of unity are included in $\OO_X$. Fix a primitive such root $\zeta=\zeta_n$ and consider the case when $\AA$ is a uniform cyclic extension of $\OO_X$. In other words, we have an invertible $\mathscr{L}$ and a section $t=t_U$ over each $U$,
such that 
$$\AA(U)=\OO_X(U)[t]/(t^n-x_U), \quad x_U\in \OO_X(U).$$

Then, with $y_i=t^i$, and $\sigma(t)=\zeta^r t$, we see that $\sigma(y_i)=\zeta^{ri}y_i$, and the product on $\W^\sigma_\AA$ becomes
\begin{equation}
\llangle \eps_i, \eps_j\rrangle =\zeta^{ri}(1-\zeta^{r(j-i)})x_U^\circlearrowright\,\,\eps_{\{i+j\!\!\!\mod n\}}, \quad i\leq j.
\end{equation}
where $(-)^\circlearrowright$ means that $(-)$ is included when $i+j\geq n$. We call the resulting hom-Lie algebra the \emph{Kummer--Witt hom-Lie algebra of level $r$} and denote it $\KW_\AA(r)$. The hom-Lie algebras $\KW_\AA(r_1)$ and $\KW_\AA(r_2)$, $r_1\neq r_2$, are in general non-isomorphic. The algebra $\KW_\AA(r)$ is called the \emph{$r$-th twist} of $\KW_\AA:=\KW_\AA(1)$. Clearly $\KW_\AA(0)$ is the abelian hom-Lie algebra. 

Observe that $\KW_\AA(\boldsymbol{\mu}_n):=\{\KW_\AA(r)\mid 0\leq r\leq n-1\}$ is the equivariant hom-Lie structure associated with $\AA$ and $G=\boldsymbol{\mu}_n$, the group (scheme) of $n$-th roots of unity. The structure $\KW_\AA(\boldsymbol{\mu}_n)$ is a $\boldsymbol{\mu}_n$-torsor in a natural way. In addition, note that we are not making any assumptions on $n$ being invertible on the base. 

The above gives immediately that $\Env(\KW_\AA)$, the \emph{Kummer--Witt algebra}, is given by the relations
\begin{equation}\label{eq:rel_Witt_finite}
\eps_i\eps_j-\zeta^{r(j-i)}\eps_j\eps_i-(1-\zeta^{r(j-i)})x_U^\circlearrowright\,\,\eps_{\{i+j\!\!\!\mod n\}}, 
\end{equation}for $i\leq j$. 
\begin{remark}
The reason we refer to (\ref{eq:Witthom}) as a Witt-hom-Lie algebra is the similarity in form and construction between this algebra and the Witt-Lie algebra. In fact, the infinitesimal hom-Lie algebras are completely analogous to the classical Witt-Lie algebra in characteristic $p$, and the algebra (\ref{eq:Witthom}) is a finite-rank analogue of the infinite-dimensional Witt-Lie algebra (sometimes called the ``centreless Virasoro algebra'') appearing in conformal field theory, for instance. 
\end{remark}
\begin{remark}The same game can clearly be played with Artin--Schreier extensions. We invite the reader to write out the corresponding relations for him/herself. 
\end{remark}

\section{Non-commutative rings from cyclic covers}\label{sec:nc_arith_cyclic}

\subsection{Polynomial identity algebras}
Let $A$ be a polynomial identity (PI) algebra. Recall that there are two disjoint subsets of $\Specm(\cent(A))$, the \emph{Azumya locus}, $\azu(A)$, and the \emph{ramification locus}, $\ram(A)$, that describe the behaviour of $A$ as a module over $\cent(A)$. A maximal ideal $\mm$ is in $\azu(A)$ if and only if $A/\mm=A\otimes_{\cent(A)}k(\mm)$ is a central simple algebra over $k(\mm)$. Then $\ram(A):=\Specm(\cent(A))\setminus \azu(A)$. It is known that $\ram(A)$ is the support of a Cartier divisor in $\Spec(\cent(A))$  (see e.g., \cite[III.2.5]{Jahnel}). 

Let $M:=A/\Mm$ and $N:=A/\mathfrak{N}$, be two simple $A$-modules, with $\Mm,\mathfrak{N}\in\Max(A)$, such that $\mm:=\Mm\cap\cent(A)=\mathfrak{N}\cap\cent(A)$. We now have the following theorem. 
\begin{thm}[M\"uller's theorem]
Let $A$ be an affine PI-algebra over a field $K$. Then 
$$\Ext^1_A(M,N)\neq \emptyset \iff \Mm\cap \cent(A)=\mathfrak{N}\cap \cent(A).$$
\end{thm}
\begin{proof}
	This is a reformulation of M\"uller's theorem as stated in \cite[Theorem III.9.2]{BrownGoodearl} using \cite[Lemma I.16.2]{BrownGoodearl}. 
\end{proof}

In other words, the ``tangent spaces'' $\Ext^1_A(M, N)$ are non-zero precisely over the ramification locus, $\ram(A)$.

We will primarily be interested in PI-algebras that are furthermore finite as modules over their centre (or central subalgebra). 

\subsection{Some ring-theoretical properties}\label{sec:ring}
The following definition is one generalisation of regularity to non-commutative rings. It is not necessary to understand the definition beyond knowing that this is a regularity property suitable for non-commutative algebraic geometry. 
\begin{dfn}\label{def:auslander} 
\begin{itemize}
	\item[(i)] Let $R$ be a ring and $M$ an $R$-module. Then the \emph{grade} of $M$ is defined as
$$j(M):=\mathrm{min}\{i\mid \mathrm{Ext}_R^i(M,R)\neq 0\}$$ or $j(M)=\infty$ if no such $i$ exists. 
	\item[(ii)] $R$ is \emph{Auslander--Gorenstein} if for every left and right Noetherian $R$-module $M$ and for all $i\geq 0$ and all $R$-submodules $N\subseteq \mathrm{Ext}_R^i(M,R)$, we have $j(N)\geq i$.
	\item[(iii)] $R$ is \emph{Auslander-regular} if it is Auslander--Gorenstein and has finite global dimension. 
	\item[(iv)] Let $R$ be a $K$-algebra, for $K$ a field. Then $R$ is \emph{Cohen--Macaulay} (CM) if $$j(M)+\mathrm{GKdim}(M)<\infty,$$ for every $R$-module $M$. Here $\mathrm{GKdim}$ denotes Gelfand--Kirillov dimension with respect to $K$.
\end{itemize}
\end{dfn}

Note that when $R$ is commutative we get the ordinary (Serre) regularity as defined in commutative algebra. 

For this section we can work slightly more generally and assume that $B$ is an admissible $k$-algebra. Recall that an \emph{admissible} (commutative) ring (or scheme) is a ring which is of finite type over a field or excellent Dedekind domain. We assume in addition that $B$ is a regular domain. Essentially everything can be made global if one is careful, but for simplicity we only consider the situation over a fixed affine patch.  The crucial difficulty arises when considering viewing Ore extensions in a global setting.

We put
$$\eW:=\frac{B\{\eps_0,\eps_1,\dots, \eps_{n-1}\}}{\left(\eps_i\eps_j-q^{-1}_iq_j\eps_j\eps_i-\sum_{k=0}^{n-1}(1-q_i^{-1}q_j)a^k_{ij}\eps_k\right)}$$ and give $\eW$ the standard ascending filtration by degree with $\mathrm{Fil}^0:=B$ and the generators $\{\eps_i\}$ in degree one. It is important to notice that, if we globalize $\eW$, we find that (\ref{eq:rel_Witt_finite}) is a special case. 

Recall that an \emph{Ore extension} (or \emph{skew-polynomial ring}) of a ring $A$ is a twisted polynomial ring $A[x;\sigma,\delta]$, where $\sigma$ is a ring morphism on $A$ and $\delta$ a $\sigma$-derivation on $A$, twisted in the sense that $xa=\sigma(a)x+\delta(a)$. 

Consider now the iterated Ore extension 
\begin{equation}\label{eq:R}R:=B[y_0][y_1;\phi_1]\cdots[y_{n-1};\phi_{n-1}],\quad\text{with} \quad\phi_i(y_j)=q_{ij}y_j,\,\,\, 0\leq j<i,
\end{equation} where we have put $q_{ij}:=q^{-1}_iq_j$ and where $q_i\in B^\times$ for all $0\leq i\leq n-1$. Notice that $R$ is in fact $\mathrm{gr}(\eW)$ with $\eW$ given the standard filtration and $y_i=\mathrm{gr}(\eps_i)$. 
\begin{prop}\label{prop:Witt_PI}Assume that the $q_i$'s are primitive $m_i$-th roots of unity. Then,
\begin{itemize}
	\item[(i)] $\mathrm{gr}(\eW)$ is a noetherian, Auslander-regular domain;
	\item[(ii)] $\mathrm{Kdim}(\mathrm{gr}(\eW))=\mathrm{Kdim}(B)+n$;
	\item[(iii)] $\mathrm{gl.dim}(\mathrm{gr}(\eW))=\mathrm{gl.dim}(B)+n$;
	\item[(iv)] we have the central subalgebra
$$B[y_0^N, y_1^N,\dots, y_{n-1}^N]\subseteq\cent(R)=\cent\big(\mathrm{gr}(\eW)\big),$$ where $N$ is the least common multiple of the $m_i$;
	\item[(v)] $R=\mathrm{gr}(\eW)$ is finite as a module over its centre and hence a polynomial identity (PI) algebra, and
	\item[(vi)] $R$ is a maximal order in its quotient ring of fractions, which is a division algebra.
\end{itemize}
\end{prop}
\begin{proof}By \cite{Ekstrom} an iterated Ore extension of a noetherian regular domain is a noetherian Auslander-regular domain, proving (i). The next two statements follow from \cite[Theorem 7.5.3]{McConnellRobson} and \cite[Proposition 6.5.4]{McConnellRobson}, respectively. Clearly, $B[y_0^N, y_1^N,\dots, y_{n-1}^N]\subseteq \cent(R)$. The whole ring $R$ is finite as a module over $B[y_0^N, y_1^N,\dots, y_{n-1}^N]$, since any monomial $y_0^{l_0}y_1^{l_1}\cdots y_{n-1}^{l_{n-1}}$ can be written as
$$y_0^{l_0}y_1^{l_1}\cdots y_{n-1}^{l_{n-1}}=\big(y_0^{l_0-Ns_0}y_1^{l_1-Ns_1}\cdots y_{n-1}^{l_{n-1}-Ns_{n-1}}\big)\cdot y_0^{Ns_0}y_1^{Ns_1}\cdots y_{n-1}^{Ns_{n-1}},$$ with $s_0, s_1, \dots s_{n-1}\geq 0$ and each $l_i-Ns_i<N$. Therefore, $R$ is finite over $\cent(R)$. From this follows that $R$ is PI by \cite[Corollary 13.1.13(iii)]{McConnellRobson}.
The last claim, that $R$ is a maximal order in its quotient ring of fractions, follows from \cite[Proposition V.2.3]{MauryRaynaud} since $R$ is prime (every Ore extension over a domain is a domain, hence prime). This is a division ring since $R$ is an Ore domain by \cite[2.1.15]{McConnellRobson} together with (i) above, and the claim then follows by \cite[2.1.14]{McConnellRobson}. 
\end{proof}

\begin{prop}\label{prop:eW}The $B$-algebra $\eW$ is an Auslander-regular, noetherian PI-domain. Consequently the ring of fractions is a division algebra. 
\end{prop}
\begin{proof}
Notice that the monomials in $\eps_0, \eps_1,\dots,\eps_{n-1}$ form a basis for $\eW$ as a $B$-module and the relations between the $\eps_i$'s are on the form 
$$\eps_i\eps_j-q_{ij}\eps_j\eps_i = \sum_{k=0}^{n-1}  a^k_{ij}\eps_k.$$ Furthermore, for $f\in B$, we have that $\eps_i f=f\eps_i$, for all $i$. Then \cite[Theorem 1 and Corollary 2]{GomezLobillo} implies that $\eW$ is Auslander-regular (at this point we could also have used that $\mathrm{gr}(\eW)$ is Auslander-regular). 
By assumption $B$ is a noetherian domain. From this, and the fact that the standard filtration is separated, follows that $\mathrm{gr}(\eW)$ is a noetherian domain if and only if $\eW$ is a noetherian domain. Since $\mathrm{gr}(\eW)$ is an iterated Ore extension of a noetherian domain it is itself a noetherian domain. To prove that $\eW$ is a PI-algebra, we note that $\eW$ is finite as a module over the commutative subalgebra $B[\eps_1^N,\eps_2^N,\dots,\eps_{n-1}^N]\subset\eW$. Hence we can conclude by \cite[Corollary 13.1.3(iii)]{McConnellRobson}. That the ring of fractions is a division algebra follows as in the proof of proposition \ref{prop:Witt_PI}. 
 \end{proof}
\begin{remark}The $B$-algebras $\eW$ are not Ore extensions in general.
\end{remark} 
 
We will not use the the following two results but I include them for interest's sake. 
Let $T$ be a commutative ring, $A$ an $T$-algebra and $M$ a finitely generated $A$-module. Then $M$ is called \emph{generically free} if there is a non zero-divisor $s\in A$ such $M[s^{-1}]:=M\otimes_A A[s^{-1}]$ is free. 
\begin{prop}Let $B$ be an admissible domain and endow $\eW$ with the standard filtration (with generators in degree one). Then every finite $\eW$-module is generically free.
\end{prop}
\begin{proof}
	We have seen that $\mathrm{gr}(\eW)$ is an Ore extension and by \cite[Proposition 4.4]{ArtinSmallZhang}, $\mathrm{gr}(\eW)$ is strongly noetherian. We can now apply \cite[Theorem 0.3]{ArtinSmallZhang} to conclude. 
\end{proof}
For a ring $B$, $K_0(B)$ is the Grothendieck group of projective $B$-modules, and $\mathrm{Pic}(B)$ is the Picard group, i.e., the group of locally free $B$-modules of rank one. 
\begin{prop}\label{prop:K_0}We have $$K_0(\eW)\simeq K_0(B),$$ and if $B=\oo_K$, the ring of integers in a number field $K$, we have
$$K_0(\eW)\simeq K_0(\oo_K)\simeq \mathrm{Pic}(\oo_K)\oplus \Z.$$
\end{prop}
\begin{proof}We filter $\eW$ with the standard filtration with $\mathrm{Fil}^0=B$ and generators in degree one. The associated graded $\mathrm{gr}(\eW)$ is flat over $B$ since it is an (iterated) Ore extension of $B$. Since $\eW$ is Auslander-regular the global dimension is finite. This implies that every cyclic $\eW$-module has finite projective dimension (e.g., \cite[7.1.8]{McConnellRobson}) and so is right regular. Therefore, the first isomorphism now follows from Quillen's theorem \cite[Theorem 12.6.13]{McConnellRobson}.

The isomorphism $K_0(\mathfrak{o}_K)=\mathrm{Pic}(\mathfrak{o}_K)\oplus \Z$, comes from the Chern character (see \cite[III.6]{Neukirch}, for instance):
$$\mathrm{ch}:\,\,\, K_0(\mathfrak{o}_K)\longrightarrow \mathrm{Pic}(\mathfrak{o}_K)\oplus \Z,\quad E\mapsto \det(E)\oplus \mathrm{rk}(E).$$ Therefore, $K_0(\eW)\simeq\mathrm{Pic}(\mathfrak{o}_K)\oplus \Z$. 
\end{proof}Notice that since $\oo_K$ is the ring of integers in $K$, $\mathrm{Pic}(\oo_K)$, is nothing but the class group of $K$. 
\begin{corollary}The morphism $K_0(\cent(\eW))\to K_0(\eW)$ induces a group morphism 
$$K_0(\cent(\eW))\to \mathrm{Pic}(\oo_K)\oplus \Z,$$via $K_0(\oo_K)$. 
\end{corollary}
\begin{remark}In order to explicitly transfer projectives between $\eW$ and $B$, it would be interesting to know an explicit isomorphism between $K_0(\eW)$ and $K_0(B)$.
\end{remark}

We say that a $B$-algebra $T$ is \emph{fibre-wise Cohen--Macaulay} if $T\otimes_B k(\pp)$ is Cohen--Macaulay for all $\pp\in\Spec(B)$.

Put $\eW_{/\pp}:=\eW \otimes_B k(\pp)$, $\pp\in\Spec(B)$. Assume that $\bar q_i\neq 0$, i.e., that $q_i\notin \pp$. Reducing the ring $R$ from (\ref{eq:R}) modulo $\pp$ gives $$R_{/\pp}:=k(\pp)[\bar y_0][\bar y_1;\bar \phi_1]\cdots[\bar y_{n-1};\bar\phi_{n-1}]$$ with $\bar\phi_i(\bar y_j)=\bar q_{ij}\bar y_j$, $\bar q_{ij}:=\bar q^{-1}_i\bar q_j$. Giving $\eW_{/\pp}$ the standard filtration with $\mathrm{Fil}^0(\eW_{/\pp})=k(\pp)$ and all generators in degree one, we find $R_{/\pp}=\mathrm{gr}(\eW_{/\pp})$.
\begin{prop}\label{prop:eW_fibre}The following holds:
\begin{itemize}
	\item[(i)] $\mathrm{GKdim}(\mathrm{gr}(\eW_{/\pp}))=\mathrm{Kdim}(\mathrm{gr}(\eW_{/\pp}))=\mathrm{gl.dim}(\mathrm{gr}(\eW_{/\pp}))=n$;
	\item[(ii)] $\mathrm{GKdim}(\eW_{/\pp})=\mathrm{GKdim}(\mathrm{gr}(\eW_{/\pp}))$;
	\item[(iii)] $\mathrm{tr.deg}(\cent(\mathrm{gr}(\eW_{/\pp})))=n$;
	\item[(iv)] $\eW$ is fibre-wise Cohen--Macaulay;
	\item[(v)] $\mathrm{gr}(\eW_{/\pp})$ is a maximal order in its division ring of fractions, and
	\item[(vi)] $\cent(\mathrm{gr}(\eW_{/\pp}))$ is an integrally closed domain.
\end{itemize}
\end{prop}
\begin{proof}
The first three statements follow from \cite[13.10.6]{McConnellRobson} and \cite[Proposition 8.1.14]{McConnellRobson}. Taking $\Lambda = k(\pp)[\bar y_0]$ in \cite[Theorem 3]{GomezLobillo} shows (iv) as the associated graded 
$$\mathrm{gr}(\Lambda)=\mathrm{gr}\big(k(\pp)[\bar y_0]\big)=k(\pp)[\bar y_0],$$with $k(\pp)[\bar y_0]$ given the standard filtration, is Cohen--Macaulay; (v) follows from \cite[Proposition V.2.3]{MauryRaynaud} again (the division ring claim follows as in proposition \ref{prop:Witt_PI}). Finally, (vi) follows from \cite[Proposition 5.1.10 b(i)]{McConnellRobson}. 
\end{proof}

\subsection{Kummer--Witt algebras}\label{sec:ncKummerWitt_spaces}
Recall the assumption that $\zeta\in B$, where $\zeta$ is a primitive $n$-th root of unity. We will consider a special case of the construction in section \ref{sec:Kummer_Witt} from which all else that follow will be built. 

As in section \ref{sec:Kummer_Witt} let $A$ be the cyclic ring extension
 	\begin{equation}\label{eq:Kummer}
 	 A:=B[t]/(t^n-x)=\bigoplus_{i=0}^{n-1} Be_i,
 	 \end{equation}with $e_i:=t^i$ and $x\in B$. Clearly, 
 	$$e_ie_j=x^\circlearrowright e_{\{i+j\!\!\!\mod n\}}.$$ This means that $A$ is a Kummer extension of $B$. The element $x\in B$ is the \emph{$A$-ramification divisor}. This element determines a canonical subscheme in the ramification locus of a non-commutative space attached to $A$. Observe that $x$ is a ramification invariant in two senses: (1) as the divisor in $B$ over which $A$ is ramified (i.e., $\pp\mid x\Rightarrow \pp$ ramified); and (2) as an element giving a subscheme of the ramification locus in a certain non-commutative space. 
	
Put $\eKW_x(r):=\Env(\KW_A(r))$. Explicitly this means that
\begin{equation}\label{eq:kum_def} 
\eKW_x(r)=\frac{B\{\eps_0,\eps_1,\dots,\eps_{n-1}\}}{\left(\eps_i\eps_j-\zeta^{r(j-i)}\eps_j\eps_i-(1-\zeta^{r(j-i)})x^\circlearrowright\,\,\eps_{\{i+j\!\!\!\mod n\}}\right)}.
\end{equation}
If $x=0$, the algebra $\eKW_0(r)$ is the enveloping algebra of an infinitesimal hom-Lie algebra over $B$ of order $n$ and so $\Xscr_{\eKW_0(r)}$ can be viewed as a non-commutative fat point of order $n$.

Observe that $\eKW_x(r)$ is a special case of the algebra $\eW$ in the previous section. 

\begin{remark}The above construction globalises immediately. Let $x:=(U_i, x_i)$ be a Cartier divisor on a scheme $X$ over a base including a $\zeta_n$. Then we define an $\OO_X$-algebra $\mathcal{W}_x^{\frac{1}{n}}(r)$ by giving it locally by (\ref{eq:kum_def}), with $x$ replaced by $x_i$. In order to keep it simple, we have opted to only write out the affine case. Every algebra, when given by generators and relations, in this paper can be globalised. Be aware, however, that the isomorphism in proposition \ref{prop:isoJ} below cannot be sheafified.  
\end{remark}

\subsubsection{Fibres of $\eKW_x(r)$}
We begin by observing that the centre commutes with taking fibres
$$\cent\Big(\eKW_x(r)\otimes_B k(\pp)\Big)=\cent\Big(\eKW_x(r)\Big)\otimes_B k(\pp).$$This follows since, if $a\otimes 1\in \cent\Big(\eKW_x(r)\otimes_B k(\pp)\Big)$ then, for $x\otimes 1\in\eKW_x(r)\otimes_B k(\pp)$, 
$$(a\otimes 1)(x\otimes 1)=ax\otimes 1\quad\text{and}\quad (x\otimes 1)(a\otimes 1)=xa\otimes 1$$ so that $ax=xa$. The other inclusion follows since any element in $\cent(\eKW_x(r))\otimes_B k(\pp)$ commutes with any element in $\eKW_x\otimes k(\pp)$. 

Let $\pp\in\Spec(B)$ be a prime. Observe that the reduction of $\zeta$ modulo $\pp$ (i.e., the image of $\zeta$ in $k(\pp)$), $\bar\zeta$, is non-zero, since $\zeta\in B^\times$. Then 
\begin{equation}\label{eq:reductionKW}
\eKW_x(r)_{/\pp}=\frac{k(\pp)\{\eps_0,\eps_1,\dots,\eps_{n-1}\}}{\left(\eps_i\eps_j-\bar\zeta^{r(j-i)}\eps_j\eps_i-(1-\bar\zeta^{r(j-i)})x^\circlearrowright\,\,\eps_{\{i+j\!\!\!\mod n\}}\right)}.
\end{equation}
We record the following for easy reference. 
\begin{prop}\label{prop:reductionKW}
We have the following three possibilities when reducing modulo a prime $\pp$:
\begin{itemize}
	\item[(1)] $\bar\zeta=1$, in which case we get
	$$\eKW_x(r)_{/\pp}=\frac{k(\pp)\{\eps_0,\eps_1,\dots,\eps_{n-1}\}}{\left(\eps_i\eps_j-\eps_j\eps_i\right)}=k(\pp)[\eps_0,\eps_1,\dots,\eps_{n-1}],$$the commutative polynomial algebra;
	\item[(2)] $\bar x =0$, in which case we get
		$$\eKW_x(r)_{/\pp}=\frac{k(\pp)\{\eps_0,\eps_1,\dots,\eps_{n-1}\}}{\left(\eps_i\eps_j-\bar\zeta^{r(j-i)}\eps_j\eps_i\right)},$$a quantum affine space;
	\item[(3)] and the generic case (\ref{eq:reductionKW}) with relations unchanged.
\end{itemize}
It is important to note that in all three cases, the reduced algebra is a domain (see proposition \ref{prop:Witt_PI}). 
\end{prop}
\begin{proof}
Obvious. 
\end{proof}

The ring-theoretic properties of $\eKW_x(r)$ are summarised in the following theorem.

\begin{thm}\label{thm:KW}The algebra $\eKW_x(r)$ satisfies the following:
\begin{itemize}
	\item[(i)] it is an Auslander-regular, noetherian PI-domain, finite over its centre;
	\item[(ii)] $\mathrm{Kdim}(\eKW_x(r))=\mathrm{gl.dim}(\eKW_x(r))= n+\mathrm{Kdim}(B)$;
	\item[(iii)] it is fibre-wise Cohen--Macaulay with $$\mathrm{GKdim}\big(\eKW_x(r)_{/\pp}\big)=\mathrm{tr.deg}\big(\eKW_x(r)_{/\pp}\big)=n;$$
	\item[(iv)] it is fibre-wise a maximal order in its fibre-wise division rings of fractions;
	\item[(v)] every finitely generated $\eKW_x(r)$-module is generically free, and
	\item[(vi)] $K_0(\eKW_x(r))\simeq K_0(B)$.
	\end{itemize}
\end{thm}
\begin{proof}
	The point (i) follows from proposition \ref{prop:eW}. By definition, $\mathrm{tr.deg}$ of a PI-algebra $S$ over a field is $\mathrm{tr.deg}(\cent(S))$. From \cite[Proposition 13.10.6]{McConnellRobson} we have 
	$$\mathrm{Kdim}(\eKW_x(r)_{/\pp})=\mathrm{tr.deg}(\eKW_x(r)_{/\pp})=\mathrm{GKdim}(\eKW_x(r)_{/\pp}).$$ Since $k(\pp)[\eps_0^n,\eps_2^n,\dots,\eps_{n-1}^n]\subseteq \cent(\eKW_x(r)_{/\pp})$, we have that 
	$$\mathrm{tr.deg}\big(\cent(\eKW_x(r)_{/\pp})\big)\geq \cent\big(k(\pp)[\eps_0^n,\eps_2^n,\dots,\eps_{n-1}^n]\big)=n,$$implying that $\mathrm{Kdim}(\eKW_x(r)_{/\pp})\geq n$. For a filtered ring $S$ we have by \cite[Lemma 6.5.6]{McConnellRobson} that $\mathrm{Kdim}(S)\leq \mathrm{Kdim}(\mathrm{gr}(S))$, and by proposition \ref{prop:eW_fibre} we have that $\mathrm{Kdim}(\mathrm{gr}(\eKW_x(r)_{/\pp}))=n$, so $\mathrm{Kdim}(\eKW_x(r)_{/\pp})=n$. The same applies to $\mathrm{GKdim}(\eKW_x(r)_{/\pp})=n$. Lifting to $B$  we get $\mathrm{Kdim}(\eKW_x(r)_{/\pp})=n+\mathrm{Kdim}(B)$. By \cite[Corollary 7.6.18]{McConnellRobson}, $\mathrm{gl.dim}(\eKW_x(r)_{/\pp})\leq \mathrm{gl.dim}(\mathrm{gr}(\eKW_x(r)_{/\pp}))$. For noetherian prime PI-rings $S$ we have that 
	$\mathrm{Kdim}(S)\leq \mathrm{gl.dim}(S)$ by \cite[Theorem 1.7(i)]{RescoSmallStafford}, so
	$$n=\mathrm{Kdim}(\eKW_x(r)_{/\pp})\leq\mathrm{gl.dim}(\eKW_x(r)_{/\pp})\leq \mathrm{gl.dim}(\mathrm{gr}(\eKW_x(r)_{/\pp}))=n,$$proving (ii) and (iii). The rest follows directly from proposition \ref{prop:eW_fibre} as $\eKW_x(r)$ is a special case of $\eW$ from that section. 
\end{proof}
\begin{remark}A natural question is to what extent the above theorem, and the results that will follow, can be extended to other ring extensions, besides Kummer extensions. 
\end{remark}
\subsubsection{Transfer of structure}\label{sec:transfer_structure}
It is obviously interesting to transfer structures such as modules or subspaces of $A_{/B}$ to corresponding structures over $\eKW_x(r)$. The following easy observation allows us to do just that. Write $\eKW_x(r)$ as $B\{\eps_0, \eps_1, \dots,\eps_{n-1}\}/I$, with $I$ the two-sided ideal of relations in $B\{\eps_0, \eps_1, \dots,\eps_{n-1}\}$ from (\ref{eq:kum_def}). We can construct a $B$-module morphism 
$$\xi:\,\, A\longrightarrow \eKW_x(r), \quad e_i=t^i\longmapsto \eps_i, \quad 0\leq i\leq n-1.$$Then, if $S\subset A$ is a subset, we can transfer $S$ to $\eKW_x(r)$ via the map $\xi$, to get the subset $\xi(S)\subset \eKW_x(r)$. In particular, if $S$ generates an ideal in $A$, $\xi(S)$ generates a two-sided ideal in $\eKW_x(r)$, and we can consider the quotient $\eKW_x(r)/\langle\xi(S)\rangle$. 

Similarly, if $M$ is an $A$-module over $B$, we have an action of $e_i$ on $M$ via some structure morphism $\rho: A\to \End_B(M)$. Via the association $\xi$ we can transfer the action of $e_i$ on $M$ to $\eps_i$ to get 
a morphism 
$$B\{\eps_0, \eps_1, \dots,\eps_{n-1}\}\xrightarrow{\rho'}\End_B(M), \quad \rho'(\eps_i)=\rho'(\xi^{-1}(e_i)):=\rho(e_i).$$Taking the invariants of $M$ under $\rho'(I)$ we get a $\eKW_x(r)$-module via the induced structure morphism 
$$\chi: \,\, \eKW_x(r)\longrightarrow \End_B(M^{\rho'(I)}).$$
\subsection{The algebra Jackson algebra $\eJac_x(r)$}\label{sec:J}
When $n>2$ there is a canonical subalgebra of $\eKW_x(r)$ that we will now study in some detail. Let me remark that at some points one needs to be a bit careful when the characteristic is two. 

First, notice that since $\zeta^n=1$, we have that $\zeta^{-(n-1)}=\zeta$ and $\zeta^{-(n-2)}=\zeta^2$. We put
$$\eJac_x(r)':=\frac{B\{\eps_0,\eps_1,\eps_{n-1}\}}{\begin{pmatrix}
 \eps_0\eps_1-\zeta^r\eps_1\eps_0- (1-\zeta^r)\eps_1\\
  \eps_{n-1}\eps_0-\zeta^r\eps_0\eps_{n-1}- (1-\zeta^r)\eps_{n-1}\\
  \eps_{n-1}\eps_1-\zeta^{2r}\eps_1\eps_{n-1}- x(1-\zeta^{2r})\eps_0
\end{pmatrix}}.$$This is clearly a subalgebra of $\eKW_x(r)$. However, it is more beneficial to work with an isomorphic algebra:
\begin{prop}\label{prop:isoJ}Let $\zeta^{2r}\neq 1$. Then the algebra $\eJac_x(r)'$ is isomorphic over $B[(1-\zeta^{2r})^{-1}]$, in particular fibre-wise, to the algebra 
	\begin{equation}\label{eq:envS}
	\eJac_x(r):=\frac{B\{\eps_0,\eps_1,\eps_2\}}{\begin{pmatrix}
 \eps_0\eps_1-\zeta^r\eps_1\eps_0,\quad
  \eps_2\eps_0-\zeta^r\eps_0\eps_2,\\
  \eps_2\eps_1-\zeta^{2r}\eps_1\eps_2- x\eps_0-x(1-\zeta^{2r})
\end{pmatrix}}
\end{equation}and $\eJac_x(r)$ is an iterated Ore extension.
\end{prop}
\begin{proof}
By changing basis $\eps_0\to (1-\zeta^{2r})^{-1}\eps_0+1$ we can transform the relations for $\eJac_x(r)'$ to the ones in (\ref{eq:envS}). Construct the iterated Ore extension
 $B[\eps_0][\eps_1;\tau][\eps_2; \tau^{-1},\delta],$ with 
 \begin{align*}
 \tau(\eps_0)&=\zeta^{-r}\eps_0, & \tau^{-1}(\eps_1)&=\zeta^{2r}\eps_1,&
 \delta(\eps_1) & =x\eps_0+x(1-\zeta^{2r}),&
 \delta(\eps_0)=0.
 \end{align*}It is easy to see that this Ore extension is isomorphic to $\eJac_x(r)$. Notice that we have implicitly extended $\tau$ to $\eps_1$ in the proof. 
 \end{proof}
The algebra $\eJac_x(r)$ is isomorphic to the (enveloping algebra of the) ``Jackson-$\mathfrak{sl}_2$'', which is a $q$-deformation of the Lie algebra $\mathfrak{sl}_2$, from \cite{LaSi}. Therefore the following definition is natural.

\begin{dfn}\label{def:jackson_space} We call the algebra $\eJac_x(r)$ the \emph{Jackson algebra} (of level $r$) associated with the cover $\Spec(A)\to \Spec(B)$.  If $r$ is irrelevant for the discussion, we will often drop it from the notation. 
\end{dfn}

\begin{remark}\label{rem:long}We make the following (long) series of remarks.
\begin{itemize}
\item[(i)] If $n=2$ the algebra $\eKW_x(r)$ only have two generators, so $\eJac_x(r)$ cannot be a subalgebra in this case. Still, abstractly, it is well-defined as given by generators and relations. 
\item[(ii)] When $\zeta^{2r}=1$, the algebra $\eJac_x(r)$ is either isomorphic to the commutative polynomial algebra $B[t_1,t_2,t_3]$ (when $\zeta=1$) or to the $B$-algebra on generators $\eps_0$, $\eps_1$, $\eps_2$ and with relations
$$\eps_0\eps_1+\eps_1\eps_0=0,\quad \eps_2\eps_0+\eps_0\eps_2=0,\quad \eps_2\eps_1-\eps_1\eps_2=x\eps_0.$$However, in this case $\eJac_x(r)'$ and $\eJac_x(r)$ are not isomorphic. 
\item[(iii)] The algebra defined by {\rm (\ref{eq:envS})} is isomorphic to the
  \emph{down-up algebra} $D_{\zeta^r}$ over $B$ (see \cite{BenkartRoby} for the definition) defined by the relations
\begin{align*}
d^2u&=\zeta^r(1+\zeta^r)dud-\zeta^{3r} ud^2+a(1-\zeta^{2r})(1-\zeta^r)d,\\
du^2&=\zeta^r(1+\zeta^r)udu-\zeta^{3r} u^2d+a(1-\zeta^{2r})(1-\zeta^r) u.
\end{align*}
To see this, solve for $a\eps_0$ in (\ref{eq:envS}) and insert in the other two relations and simplify.
\item[(iv)] If $x=0$, we get, after a renaming of generators, the same relations as in (\ref{eq:Aq3_inf}) and hence a quantum $\mathbb{A}^3$.
\item[(v)] The algebras $\eKW_x(r)$ are \emph{not} Ore extensions in general. 
\item[(vi)] Notice that $\eJac_x(r)$ includes two copies of the quantum plane $\mathbf{Q}_B(r):=B\{t_1, t_2\}/(t_1t_2-\zeta^r t_2t_1)$. Hence $\eJac_x(r)$ is constructed by ``glueing'' the quantum planes via the third relation in (\ref{eq:envS}).  
\item[(vii)] The algebra $\eJac_x/(\eps_0)$ is the first quantum Weyl algebra.
\end{itemize}
\end{remark} The following deserves its own remark:
\begin{remark}Since $A_{/B}$ is a Kummer extension and the ramification properties of such extensions are intimately related to the divisor $x$, it is natural to assume that the ramification is related to the algebra $\eJac_x$. This is indeed the case, and one of the reasons I started this project. It will be a fundamental part of a sequel to the present paper to study this connection in more depth. 
\end{remark}
We give the relations in the cases $n=3$ and $n=4$.
\begin{example}When $n=3$ we get
\begin{align*}
r=0:&\quad\begin{Bmatrix}
	\eps_0\eps_1-\eps_1\eps_0&=&0\\
	\eps_2\eps_0-\eps_0\eps_2&=&0\\
		\eps_2\eps_1-\eps_1\eps_2&=&x\eps_0 
	\end{Bmatrix}\\
r=1:&\quad\begin{Bmatrix}
	\eps_0\eps_1-\zeta_3\eps_1\eps_0&=&0\\
	\eps_2\eps_0-\zeta_3\eps_0\eps_2&=&0\\
	\eps_2\eps_1-\zeta_3^{2}\eps_1\eps_2&=&x\eps_0+x(1-\zeta_3^{2})\end{Bmatrix}\\
r=2:&\quad \begin{Bmatrix}\eps_0\eps_1-\zeta_3^2\eps_1\eps_0&=&0\\
 \eps_2\eps_0-\zeta_3^2\eps_0\eps_2&=&0\\
\eps_2\eps_1-\zeta_3\eps_1\eps_2&=&x\eps_0+x(1-\zeta_3)\end{Bmatrix}
\end{align*}
The different cases are non-isomorphic. Observe that the case $r=0$ is in fact the universal enveloping algebra of a solvable $3$-dimensional Lie algebra. 
\end{example}
\begin{example}In a sense the case $n=4$ is more interesting. Recall that $\zeta_4$ is chosen to be primitive. This means in particular that $\zeta_4^2$ is not primitive, hence the case $r=2$ is rather special. 
\begin{align*}
r=1:&\quad\Big\{\eps_0\eps_1-\zeta_4\eps_1\eps_0=0,\quad \eps_2\eps_0-\zeta_4\eps_0\eps_2=0,\quad\eps_2\eps_1+\eps_1\eps_2=x\eps_0+2x\Big\}\\
r=2:&\quad\Big\{\eps_0\eps_1+\eps_1\eps_0=0,\quad \eps_2\eps_0+\eps_0\eps_2=0,\quad\eps_2\eps_1-\eps_1\eps_2=x\eps_0\Big\}\\
r=3:&\quad\Big\{\eps_0\eps_1+\zeta_4\eps_1\eps_0=0,\quad \eps_2\eps_0+\zeta_4\eps_0\eps_2=0,\quad\eps_2\eps_1+\eps_1\eps_2=x\eps_0+2x\Big\}.
\end{align*}Obviously, the case $r=0$ is the same for all $n$. 
\end{example}

\subsection{Ring-theoretic and geometric properties of $\eJac_x(r)$}\label{sec:centreBS}
\subsubsection{The centre}
We will use a result of A. D. Bell and S.P. Smith from \cite{BellSmith}. Unfortunately, as far as I'm aware, this result is not publicly available so, for completeness, I include their proof. Therefore, except for corollary \ref{cor:BellSmith}, there is nothing original (apart for minor modifications) in the following section. Any mistakes are certainly my own. 

To be consistent with the notation in Bell and Smith's work we rearrange the last relation in $\eJac_x(r)$ :
\begin{equation}\label{eq:envS2}
	\eJac_x(r)=\frac{k(\pp)\{\eps_0,\eps_1,\eps_2\}}{\begin{pmatrix}
 \eps_0\eps_1-\zeta^r\eps_1\eps_0,\quad
  \eps_2\eps_0-\zeta^r\eps_0\eps_2,\\
  \eps_1\eps_2-\zeta^{-2r}\eps_2\eps_1- a\eps_0-b
\end{pmatrix}},
\end{equation}
where we have put $a:=-x\zeta^{-2r}$ and $b:=-x\zeta^{-2r}(1-\zeta^{2r})=x(1-\zeta^{-2r})$, for simplicity.

Put $\wps:=\eps_1\eps_2$ and $W:=k(\pp)[\eps_0, \wps]$. Then $W$ is a commutative subalgebra of $\eJac_x(r)$ and $\eJac_x(r)$ is free as a module over $W$. 

Define $$\sigma(\eps_0)=\zeta^r\eps_0\quad\text{and}\quad \sigma(\wps)=\zeta^{2r}(\wps-a\eps_0-b)=\eps_2\eps_1.$$ Then $\sigma\in\Aut(W)$ (this is part of \cite[Lemma 3.2.2]{BellSmith}). One sees immediately that 
\begin{equation}\label{eq:dsigma}
w\eps_1 =\eps_1\sigma(w)\quad\text{and}\quad \eps_2w=\sigma(w)\eps_2\quad \text{for all $w\in W$}. 
\end{equation}

Put $\eJac_x(r)^n:=W\eps_2^n=\eps_2^n W$ and $\eJac_x(r)^{-n}:=W\eps_1^n=\eps_1^n W$. Then 
$$\eJac_x(r)=\bigoplus_{n\in\Z} \eJac_x(r)^n=\eJac_x(r)^-\bigoplus \eJac_x(r)^+,$$with  
$$\eJac_x(r)^-:=\bigoplus_{n<0}\eJac_x(r)^n, \quad \eJac_x(r)^+:=\bigoplus_{n\geq 0}\eJac_x(r)^n.$$One can easily show that $\eJac_x(r)^+=W[\eps_2; \sigma]$ and $\eJac_x(r)^-=W[\eps_1; \sigma^{-1}]$ are both Ore extensions. The following proposition is part of proposition 3.2.4 in Bell and Smith \cite{BellSmith}.  
\begin{prop}\label{prop:BellSmith}Put $s:=\mathrm{ord}(\sigma)$. Observe that $s$ must be such that $sr$ is a multiple of $n$ by the definition of $\sigma$. Then 
$$\cent(\eJac_x(r))=W^\sigma[\eps_1^s, \eps_2^s],$$where $W^\sigma$ is the invariant subring under $\sigma$.
\end{prop}
\begin{proof}
	An element in $\eJac_x(r)$ is central if and only if every term is. Let $w\in W$ and consider $w\eps_2^i$. Then for $w'\in W$ we have, by (\ref{eq:dsigma}),
	$$w\eps_2^iw'=\sigma(w')^iw\eps_2^i$$ and, since $\eJac_x$ is a domain, $w\eps_2^i$ commutes with all $w'\in W$ if and only if $\sigma^i=\id$. We similarly see that $[\eps_2, w\eps_2^i]\iff \sigma(w)=w$. 
	
	Suppose now that $\sigma^i=\id$ and that $\sigma(w)=w$. Then, since $\eJac_x$ is a domain, $w\eps_2^i$ commutes with $\eps_1$ if and only if 
	\begin{equation}\label{eq:e1e2w}
	\eps_2\eps_1 w\eps_2^i=\eps_2 w\eps_2^i\eps_1.
	\end{equation}  Now, 
	$$\eps_2\eps_1 w\eps_2^i=\sigma(\wps)w\eps_2^i=w\sigma(\wps)\eps_2^i$$ and 
	\begin{align*}
	\eps_2 w\eps_2^i\eps_1&=w\eps_2^{i+1}\eps_1=w\eps_2^i\eps_2\eps_1=w\eps_2^i\sigma(\wps)
	=w\sigma^{i+1}(\wps)\eps_2^i=w\sigma(\wps)\eps_2^i.
	\end{align*}Hence, (\ref{eq:e1e2w}) is proven and so
	$$	w\eps_2^i\in \cent(\eJac_x)\iff \sigma^i=\id\text{ and } \sigma(w)=w.$$The exact same reasoning applies to $w\eps_1^i$, thereby completing the proof. 
\end{proof}
From this follows:
\begin{corollary}\label{cor:BellSmith}We have, still with $\wps=\eps_1\eps_2$, 
$$W^{\sigma}=\begin{cases}
	k(\pp)[\eps_0^l], &\text{if $a, b\neq 0$}\\
	k(\pp)\big[\eps_0^l, \wps^l\big], &\text{if $a=b=0$}
	\end{cases}$$and so
	$$\cent(\eJac_x(r))=\begin{cases}
	k(\pp)\big[\eps_0^l, \eps_1^l, \eps_2^l\big], &\text{if $a, b\neq 0$}\\
	k(\pp)\big[\eps_0^l, \eps_1^l, \eps_2^l, \wps^t\big], &\text{if $a=b=0$},
	\end{cases}$$where $l$ is the least integer such that $lr\equiv 0\,\, (\mathrm{mod}\,\, n)$ and $t$ minimal with the property that $2tr\equiv 0\,\,(\mathrm{mod}\,\, n)$. 
\end{corollary}
Observe that the congruences imply that $2tr\equiv lr\,\, (\mathrm{mod}\,\, n)$, i.e., $n\mid(2t-l)r$.
\begin{proof}
	We begin by determining $W^\sigma$. First, that $\eps_0^m\in W^\sigma$ means that 
$\sigma(\eps_0^m)=\zeta^{mr}\eps_0^m$, so $mr\equiv 0\,\, (\mathrm{mod}\,\, n)$. We can thus assume that $m=l$, the least such integer such that $lr\equiv 0\,\, (\mathrm{mod}\,\, n)$. An induction argument shows that
$$\sigma^k(\wps)=\zeta^{2kr}\wps-a\zeta^{(k+1)r}[k]_{\zeta^r}\eps_0-\zeta^{2r}[k]_{\zeta^{2r}}b,$$ and from this, together with $\sigma^k(\eps_0)=\zeta^{kr}\eps_0$,  follows that the order of $\sigma$ must be $l$. Let $\eps_0^l\wps^t\in W$. Then,
\begin{align*}
	\sigma(\eps_0^l\wps^t)&=\sigma(\eps_0)^l\sigma(\wps)^t=\zeta^{lr}\eps_0^l\big(\zeta^{2r}(\wps-a\eps_0-b)\big)^t\\
	&=\zeta^{(2t+l)r}\eps_0^l(\wps-a\eps_0-b)^t.
\end{align*} If this element shall be invariant under $\sigma$ we must thus have
$$\zeta^{(2t+l)r}\eps_0^l(\wps-a\eps_0-b)^t=\eps_0^l\wps^t,$$which, since $W$ is a domain, is equivalent to 
\begin{equation}\label{eq:winvariant}
\zeta^{(2t+l)r}(\wps-a\eps_0-b)^t=\wps^t.
\end{equation}The trinomial identity allows us to expand the parentheses in the left-hand-side to obtain the condition
$$\zeta^{(2t+l)r}\sum_{i+j+k=t}(-1)^{j+k}\frac{t!}{i!j!k!}a^jb^k\eps_0^j\wps^i=\wps^t.$$From this follows that if $a,b\neq 0$, this can never occur unless $t=0$. 

If $a=b=0$, then (\ref{eq:winvariant}) implies that $\zeta^{(2t+l)r}\wps^t=\wps^t$. Since $lr\equiv 0\,\,(\mathrm{mod}\,\, n)$ we see that we must have $2tr\equiv 0\,\,(\mathrm{mod}\,\, n)$. Take, $t$ minimal (possibly zero) with this property. Then proposition \ref{prop:BellSmith} shows the claim concerning the centre, thereby completing the proof. 
\end{proof}
%
%

\subsubsection{Algebraic geometry of $\eJac_x(r)$}
\begin{thm}\label{prop:centre_J}
The centre of $\eJac_x(r)$ is given fibre-wise as follows.
\begin{itemize}
	\item[(i)] For $x=0$:
	$$\cent(\eJac_0(r))_{/\pp}=k(\pp)\big[\eps_0^l, \eps_1^l, \eps_2^l, \wps^t\big],$$
	where $l$ is the least integer such that $lr\equiv 0\,\, (\mathrm{mod}\,\, n)$ and $t$ minimal with the property that $2tr\equiv 0\,\,(\mathrm{mod}\,\, n)$, i.e., $t$ is minimal such that $2tr$ is a multiple of $n$. 
	\item[(ii)] For $x\neq 0$:
	$$\cent(\eJac_x(r))_{/\pp}=k(\pp)[\eps_0^l, \eps_1^l, \eps_2^l]$$with $l\in\N$ as in (i).
\end{itemize}
In both cases we have
\begin{itemize}
	\item[(iii)] $\eJac_x(r)$ is an Auslander-regular, noetherian, fibre-wise Cohen--Macaulay domain, finite as a module over its centre and hence a polynomial identity ring (PI) of $\mathrm{pideg}(\eJac_x(r))=n$;
	\item[(iv)] $\eJac_x(r)$ is a maximal order in its division ring of fractions;
	\item[(v)] $K_0(\eJac_x(r))\simeq K_0(B)$;
	\item[(vi)] $\Spec(\cent(\eJac_x(r)))$ is a normal, irreducible scheme of dimension three, for all $x$. 
	\item[(vii)] $\Spec(\cent(\eJac_x(r)))$ is in addition fibre-wise Cohen--Macaulay in the commutative sense, i.e., 
	$$\Spec(\cent(\eJac_x(r)_{/\pp}))=\Spec\big(\cent(\eJac_x(r))\otimes_B k(\pp)\big)$$ is a Cohen--Macaulay scheme for all $\pp\in\Spec(B)$ and all $x$. 
\end{itemize}
\end{thm}
\begin{proof}Properties (i) and (ii) is included in corollary (\ref{cor:BellSmith}). 

Continuing with (iii), the same argument as in the proof of proposition \ref{prop:Witt_PI} shows that $\eJac_x(r)$ is finite over its centre, and from this follows that $\eJac_x(r)$ is PI by \cite[Corollary 13.1.13(iii)]{McConnellRobson} and \cite[Proposition V.2.3]{MauryRaynaud} once again shows that it is a maximal order in its division ring of fractions (as in proposition \ref{prop:Witt_PI} again). Since $\eJac_x(r)$ is an iterated Ore extension over a noetherian domain, it is Auslander-regular and a noetherian domain itself. 

As for the pi-degree, it looks at a first glance like the rank of $\eJac_x(r)$ over its centre is $n^3$, which it obviously cannot be (it must be a square). On the other hand, on closer inspection of the relations involved, we see that $\eps_0$ can be eliminated from $\eJac_x(r)$ over $\cent(\eJac_x(r))$. Hence we effectively only have two generators, and so $\mathrm{pideg}(\eJac_x(r))=n$ also in this case.  

Part (v) follows by the same argument as in proposition \ref{prop:K_0} and part (vi) follows from \cite[5.1.10 b(i)]{McConnellRobson} and (iv) above. That the dimension is three is a consequence of \cite[Proposition 13.10.6]{McConnellRobson}. Finally, when $x = 0$ the claim concerning the Cohen--Macaulayness in (vii) follows from \cite[Lemma 2.2]{Ceken_et_al} and the case $x\neq 0$ is obvious by (ii). 
\end{proof}
\begin{corollary}With the notation as above:
\begin{itemize}
	\item[(i)] The algebra $\eJac_x(r)$ is finite as a module over $k(\pp)\big[\eps_0^l, \eps_1^l, \eps_2^l\big]$.

	\item[(ii)] Hence, the ring extension $k(\pp)\big[\eps_0^l, \eps_1^l, \eps_2^l\big]\subseteq \cent(\eJac_x(r))$ is finite, and consequently the morphism 
$$\psi: \,\,\Spec\big(\cent(\eJac_x(r))\big)\to \mathbb{A}^3_{(l)}:=\Spec\Big(k(\pp)\big[\eps_0^l, \eps_1^l, \eps_2^l\big]\Big)$$ is finite as a morphism of schemes. We put $l$ in the notation to indicate that we have a weighted version of the affine three-space. 
	\item[(iii)]  Any maximal ideal $\Mm$ in $\eJac_x(r)$ intersects the centre uniquely at a maximal ideal $\mm$.
	\item[(iv)] In the other direction, any maximal $\mm$ in $k(\pp)\big[\eps_0^l, \eps_1^l, \eps_2^l\big]$ splits into $i$ maximal ideals in $\eJac_x(r)$, where $1\leq i\leq m$ and where $m$ is the rank of $\eJac_x(r)$ as a module over $k(\pp)\big[\eps_0^l, \eps_1^l, \eps_2^l\big]$.
\end{itemize}   
\end{corollary}
\begin{proof}The first claim is clear from the theorem and the rest then follows from (i).
\end{proof}
It is in fact quite easy to find an explicit presentation of the scheme in theorem \ref{prop:centre_J} (i).
\begin{thm} Put $u_i:=\eps_i^l$ for $0\leq i\leq 2$, and $u_3:=\wps^t$. Then
	$$\cent(\eJac_0(r))_{/\pp}=k(\pp)\big[u_0, u_1, u_2, u_3\big]\Big/\big(u_3^r-u_1^au_2^a\big),$$where $a$ is the minimal integer such that $tr=an$. From this also follows that $t=l$ and that $r$ and $a$ cannot both include a factor $2$. The integer $a$ is uniquely determined by $n$ and $r$. 
\end{thm}
\begin{proof}
A simple induction argument shows that $\wps^k=\zeta^{k(k-1)r}\eps_1^k\eps_2^k$, for all $k\in\N$. Recall that $t$ is the minimal integer such that $2tr \equiv 0\,\, (\mathrm{mod}\,\, n)\iff 2tr=sn$, for some $s$. We must have that $s=2a$ for some $a$, since otherwise $u_1$ and $u_2$ are not defined (we lose commutativity). Therefore $tr=an$ and so,
$$\wps^{tr}=\zeta^{tr(tr-1)r}\eps_1^{tr}\eps_2^{tr}\iff (\wps^t)^r=\zeta^{an(an-1)}\eps_1^{an}\eps_2^{an}\iff u_3^r=u_1^au_2^a.$$This also proves that $t=l$ by minimality of $t$ and that $r$ and $a$ cannot both include a factor $2$, since otherwise the centre would not be a domain. 
\end{proof}It is now easy to convince oneself of the validity of the following corollary:
\begin{corollary}Let $X_r$ be the family of affine surfaces
$$X_r:=\Spec\left(\frac{k(\pp)[u_1, u_2, u_3]}{\big(u_3^r-u_1^au_2^a\big)}\right). $$Then  
$$\Spec\big(\cent(\eJac_0(r))_{/\pp}\big)=\mathbb{A}^1\times X_r.$$ The surface $X_r$ furthermore satisfies:
\begin{itemize}
	\item[(i)] $X_1$ is a regular and rational.
	\item[(ii)] $X_r$, for  $r>1$, are ramified $r$-covers of $\mathbb{A}^2$, with branch locus the coordinate axes $u_1=u_2=0$, and singular at the origin. 
\end{itemize}
Observe that $\Spec\big(\cent(\eJac_0(r))_{/\pp}\big)$ is a trivial $\mathbb{A}^1$-fibration for all $r$ and that $a$ is uniquely determined by $r$ and $n$. 
\end{corollary}
We can also prove that the singularities in the above corollary are rational:
\begin{thm}\label{thm:rationalCM}Let $K$ be a field of characteristic zero and let $\eJac_x(r)^\mathrm{al}_{/K}$ be the base change $\eJac_x(r)_{/K}\otimes_K K^\mathrm{al}$ of $\eJac_x(r)_{/K}$ to the algebraic closure $K^\mathrm{al}$. Then $\Spec\big(\cent(\eJac_x(r)^\mathrm{al}_{/K})\big)$ has \emph{rational} singularities for all $x$. Hence $\eJac_x(r)$ has rational singularities on the generic fibre over $\Spec(B)$ (recall that $B$ is a domain). 
\end{thm}
\begin{proof}When $x\neq 0$ the centre is $\mathbb{A}^3$ (with some weight) by theorem \ref{prop:centre_J} (ii) so we can assume that $x=0$. 

We know that $\eJac_x(r)^\mathrm{al}_{/K}$ is Auslander-regular and Cohen--Macaulay by theorem \ref{prop:centre_J}(iii). This implies that $\eJac_x(r)^\mathrm{al}_{/K}$ is \emph{homologically homogeneous} (which we won't define here) by \cite[Corollary 3.8(ii) and Note 3.5(i)]{Zhong}. Note that for affine pi-algebras, being \emph{Macaulay} (as is discussed in this reference) is equivalent to being \emph{Cohen}--Macaulay, and being Macaulay implies being \emph{locally Macaulay}.  
	
	Now the main theorem of \cite{StaffordVandenBergh} implies that $\cent\big(\eJac_x(r)^\mathrm{al}_{/K}\big)$ has rational singularities.  
\end{proof}
There's probably an easier way to show rationality, simply by looking at the coordinate ring. Also, it could be that the singularities are rational even before going to the algebraic closure. 

\section{Geometry of $\Xscr_{\eJac_x}\to \mathbb{A}^3_{(l)}$}\label{sec:fibre}
Even though we haven't formally defined what should be meant by a ``non-commutative scheme'', we will use the language of schemes in what follows. The following is a ``soft'' definition of a non-commutative scheme. For the rigorous definition see \cite{LarssonAritGeoLargeCentre}. 

Let $A$ be a $B$-algebra, where $B$ is a commutative ring. Then we define the \emph{non-commutative scheme} or \emph{non-commutative space} of $A$ to be
$$\Xscr_A:=\big(\Mod(A), \OO_A\big), \quad \OO_A:=A.$$ We often identify $M$ with its annihilator.

Let $\fam{M}:=\{M_1, M_2, \dots, M_r\}$ be a family of $A$-modules. Then the \emph{tangent space of $\Xscr_A$ at $\fam{M}$} is the collection of $\Ext^1$-groups
$$T_\fam{M}:=\Big\{\Ext^1_A(M_i, M_j)\,\,\big\vert\,\, 1\leq i, j\leq r\Big\}.$$The tangent space $T_\fam{M}$ controls the simultaneous non-commutative deformations of the modules, \emph{as a family}. 

The ring object $\OO_A$, which we simply have put equal to $A$ here, is actually built from $T_\fam{M}$ via matric Massey products, by taking the projective limits over all families of $A$-modules. See \cite{EriksenLaudalSiqveland} or \cite{LarssonAritGeoLargeCentre}. For simplicity we view $\OO_A$ as a global object rather than the as a local object which it actually is. Hence the assignment $\OO_A:=A$. 

\begin{dfn}\label{dfn:rationalpoint}
Suppose $A$ is a $K$-algebra with $K$ a field, $L/K$ a field extension and $M$ an $L$-vector space. Then an \emph{$L$-rational point} on $\Xscr_A$ is a $K$-linear algebra morphism $\rho:A\to\End_L(M)$ such that $\ker\rho$ is a maximal ideal. The set of all $L$-rational points are denoted $\Xscr_A(L)$.
\end{dfn}
\subsection{A family of divisors of degree 2}\label{sec:subspace_deg2}
Let $B$ be a commutative Dedekind domain with $\zeta:=\zeta_n\in B$. 
\begin{prop}Put $\aa:=(a_0,a_1,a_2)\in \mathbb{A}^3_{/B}$. 
\begin{itemize}
	\item[(a)] The elements $$\Omega_{\aa}:=a_0\eps_2\eps_1-a_1\eps_1\eps_2+a_2\eps_0^2-x\frac{a_0-a_1\zeta^{-1}}{1-\zeta^r}\eps_0-x(a_0-a_1),$$ defines a family of normal elements in $\eJac_x(r)$, parametrised by $\mathbb{A}^3$. 
\item[(b)] In fact, $\Omega_{\aa}$ defines an automorphism $\gamma$ of $\eJac_x(r)$ as $ 
\Omega_{\aa} \cdot t=\gamma(t)\cdot\Omega_{\aa}$ with
$$\gamma(\eps_0)=\eps_0,\quad \gamma(\eps_1)=\zeta^{2r}\eps_1, \quad \gamma(\eps_2)=\zeta^{-2r}\eps_2. $$
\item[(c)] Hence, we can view $\Omega_\aa$ as defining a flat family of non-commutative quadric surfaces 
$$Y_\aa:=\Xscr_{\eJac_x(r)/\langle\Omega_{\aa}\rangle}$$ embedded in $\Xscr_{\eJac_x(r)}$ and parametrised by $\aa\in\mathbb{A}^3$. 
\item[(d)] The spaces $Y_\aa$ (i.e., the algebras $\eJac_x(r)/\langle\Omega_{\aa}\rangle$) are Auslander-regular for all $\aa\in\mathbb{A}^3$. 
\end{itemize}	
\end{prop}
\begin{proof}
Straightforward (but long) computations show (a) and (b) and then point (c) follows directly from definition. It is well-known that a quotient of an Auslander-regular algebra by a normal element is Auslander-regular, showing (d). 
\end{proof}
The intersection with the centre, i.e., the image of $\phi$, is given by the (commutative) quadric family
$$\cent(\Omega_{\aa})=(a_0-a_1)u_1u_2+a_2u_0^2-x\frac{a_0-a_1\zeta^{-1}}{1-\zeta^r}u_0-x(a_0-a_1),$$with $u_0:=\eps_0^l$, $u_1:=\eps_1^l$ and $u_2:=\eps_2^l$, $l$ being the least integer such that $lr=n$. This means that we have a fibration $$\phi_{\aa}:=\phi\vert_{Y_{\aa}}:\,\,Y_{\aa}\to \Spec\big(\cent(\eJac_x(r))/\langle\cent(\Omega_{\aa})\rangle\big)\to \Spec(B).$$Notice that 
$P_\aa:=\Spec\big(\cent(\eJac_x(r))/\langle\cent(\Omega_{\aa})\rangle\big)$ is an affine quadric, and $Y_\aa$ is an non-commutative space over this surface, defined by the order $\eJac_x(r)/\langle\Omega_\aa\rangle$. The quadric $P_\aa$ is the ``commutative shadow'' of $Y_\aa$. 

\begin{remark}
If $\Omega_{\aa'}$ and $\Omega_{\aa''}$ are two normal elements in the above family, then $\Omega_{\aa'}\Omega_{\aa''}$ is also a normal element. Therefore, it is possible to construct more complicated subspaces (of higher degree) from elements in the family $\Omega_{\aa}$.
\end{remark} 
\subsection{Fibres over $\mathbb{A}^3_{(n)}$}\label{sec:fibres_A}
\subsubsection{Geometric and arithmetic fibres}
There is an important distinction to make concerning the notion of ``fibre'' in the non-commutative context. Let $A$ be an $R$-algebra with $R\subseteq \cent(A)$. Contraction of prime ideals defines a morphism 
$$\alpha:\,\, \Xscr_A\to\Spec(R).$$ Let $\pp\in\Spec(R)$. We can then talk of the ``fibre of $\alpha$ over $\pp$'' as
$$\alpha^{-1}(\pp):=\Xscr_{A\otimes_R k(\pp)}.$$In the context of PI-algebras there are two types of fibres depending on whether $\pp\in\azu(A)$ or $\pp\in\ram(A)$. If $\pp\in\azu(A)$, then $A_{/\pp}:=A\otimes_R k(\pp)$ is a central simple algebra over $k(\pp)$ and so $\Xscr_{A_{/\pp}}$ is one simple module as, both $A_{/\pp}$-module and $A$-module. 

On the other hand, if $\pp\in\ram(A)$, the ``fibre algebra'' $A_{/\pp}$ is not central simple. However, it is certainly artinian so  $A_{/\pp}/\mathrm{rad}(A_{/\pp})$, where $\mathrm{rad}$ is the Jacobson radical, is semi-simple and $A_{/\pp}$ and $A_{/\pp}/\mathrm{rad}(A_{/\pp})$ have the same simple modules. The point here is that, even if $\pp\in\ram(A)$, it can happen that $A_{/\pp}/\mathrm{rad}(A_{/\pp})$, and hence also $A_{/\pp}$, only has one simple module. This apparent contradiction is resolved by being careful what algebra we mean: there is only one simple module of $A_{/\pp}$ as \emph{$A_{/\pp}$-module}, but as \emph{$A$-module} there are more. 

Therefore, when we speak of ``fibre'' when need to be careful what we mean: do we mean the points of $\alpha^{-1}(\pp)$ as $A$-modules or as $A_{/\pp}$-modules. We say that $\alpha^{-1}(\pp)=\Xscr_{A_{/\pp}}$ is the \emph{geometric fibre} at $\pp$, i.e., we view $\alpha^{-1}(\pp)$ as a set of $A_{/\pp}$-modules; the set 
$$\Phi(\pp):=\Big\{\mathfrak{P}\in\Xscr_A\,\,\,\big\vert\,\,\, \mathfrak{P}\cap R=\pp\Big\}$$ is the \emph{arithmetic fibre}. By the \emph{fibre} of $A$ at $\pp$, we mean the \emph{algebra} $A_{/\pp}=A\otimes_Rk(\pp)$.

\subsubsection{Ramification over a central subalgebra}
Suppose now that $A$ is a PI-algebra, finite over a central subring $R$ and consider the inclusions 
$$R\xrightarrow{f}\cent(A)\xrightarrow{g} A.$$We say that $A$ is \emph{ramified over $R$} if either $f$ is a ramified map (i.e., the extension $R\to \cent(A)$ is ramified) or $\ram(A)\neq \emptyset$ (or both). The \emph{ramification locus} of $g\circ f$ is 
$$\ram_R(A):=\ram(g\circ f):=\branch(f)\cup \big(\ram(A)\cap R\big).$$The complement of $\ram(g\circ f)$ is the \emph{azumaya locus}, $\azu_R(A):=\azu(g\circ f)$, of $g\circ f$. Here we see that slightly unfortunate clash between the algebraic-geometric and the PI-theoretic notions of ramification. 

Notice that, even if $\pp\in\azu_R(A)$ there can be more than one prime of $A$ over $\pp$ since the $\pp$ might split in $\cent(A)$. 
\subsubsection{The fibres}
For the rest of this section we work over a field $B=K$. Recall the assumption that $\zeta\in K$ is primitive. It is important to note that we, to simplify the discussion, now only consider the case $r=1$. Hence $l=n$ and we write 
$$\eJac_x:=\eJac_x(1)=\frac{K\{\eps_0,\eps_1,\eps_2\}}{\begin{pmatrix}
 \eps_0\eps_1-\zeta\eps_1\eps_0,\quad
  \eps_2\eps_0-\zeta\eps_0\eps_2,\\
  \eps_2\eps_1-\zeta^{2}\eps_1\eps_2- x\eps_0-x(1-\zeta^{2})
\end{pmatrix}}.$$The characteristic $p$ of $K$ is arbitrary.

We begin by recalling that a \emph{symbol algebra} or \emph{cyclic algebra} (see e.g., \cite[Corollary 2.5.5]{GilleSzamuely}) over a field $K$ (or a commutative ring) is an algebra $(a,b)_{\xi}$ with generators $x$ and $y$ such that $$xy-\xi yx=0\quad\text{and} \quad x^n=a,\,\, y^n=b\quad\text{where}\quad a,b\in K^\times,\,\, \xi^n=1,$$ with $n$ invertible in $K$ (in particular $n\neq p$). Symbol algebras are central simple algebras. 

The actual construction goes as follows. Let $L$ be the cyclic extension $L:=K[t]/(t^n-a)$. Since $L$ is cyclic over $K$, there is $\sigma\in\Gal(L/K)$ such that $\sigma(t)=\xi t$. Then, for $c\in K$, $(a, c)_\xi$ is the $L$-vector space
\begin{equation}\label{eq:symbol}
(a,c)_\sigma= (a,c)_\xi := \bigoplus_{i=0}^{n-1} L\ee^i, \quad\text{with } \ee^n=c\quad\text{and }  \ee t= \sigma(t)\ee= \xi t\ee.
\end{equation}Clearly this is the quotient of the Ore extension $\big(K[t]/(t^n-a)\big)[\ee; \sigma]$ by the central element $\ee^n-c$. 

From \cite[Exercise 4.10]{GilleSzamuely} we have that $$(a,b)_{\xi}\simeq (a,c)_{\xi}\iff b^{-1}c\in\mathrm{Nm}(L),$$ where $\mathrm{Nm}$ is the ordinary norm function $\mathrm{Nm}: L\to K$. It also turns out that $L$ is a splitting field for $(a, b)_\xi$, i.e., $(a, b)_\xi\otimes_K L\simeq M_r(L)$ for some $r\geq 1$. 

Put $$\Lambda:=K[u_0, u_1, u_2]=K[\eps_0^n, \eps_1^n, \eps_2^n],\quad 
Z:=\Spec(\cent(\eJac_x)),\quad \mathbb{A}^3_{(n)}:=\Spec(\Lambda)$$ and denote by $\psi$ the scheme morphism
$$\psi: Z\longrightarrow \mathbb{A}^3_{(n)}$$induced by restriction of primes. Remember that $\psi$ might actually be the identity in some cases, depending on $Z$. 

Pick a maximal ideal $$\mm = \mm(a,b,c) =(u_0-a,u_1-b,u_2-c), \quad a,b,c\in k(\mm),$$in $\Lambda$. The fibre (as algebra) of $\eJac_x$ over $\mm$ is 
$$J_\mm:=\eJac_x/\mm=\eJac_x\otimes_{\Lambda} k(\mm)=\frac{\eJac_x}{(\eps_0^n-a,\eps_1^n-b,\eps_2^n-c)}.$$

Recall from remark \ref{rem:long}(6) that $\eJac_x$ includes two copies of the quantum plane $\mathbf{Q}_K=\mathbf{Q}_K(1)$. Recall also that $\mathbf{Q}_K$ is independent on $x$. It is well-known that the centre of $\mathbf{Q}_K$ is $K[x^n, y^n]$ (if we use $x$ and $y$ as generators for $\mathbf{Q}_K$) and that $\ram(\mathbf{Q}_K)$ is the union of the coordinate axes. Hence reduction modulo a maximal ideal $\mm'\in\Max(\cent(\mathbf{Q}_K))$ gives an algebra of the type
\begin{equation}\label{eq:quantumplane}
(a,b)_\zeta:=\mathbf{Q}_K\otimes_{\Lambda}k(\mm')=\frac{k(\mm')\{x, y\}}{\big(x^n-a, y^n-b, xy-\zeta xy\big)}.
\end{equation}In fact we can consider this from another angle.  Indeed, we can view $(a, b)_\zeta$ as the fibres over $\azu(\mathbf{Q}_K)$ of the inclusion $\cent(\mathbf{Q}_K)\to\mathbf{Q}_K$. Over $\ram(\mathbf{Q}_K)$, the fibres $(a,b)_\zeta$ become non-commutative fat points (and certainly not central simple). 

Consider the following diagram of inclusions of algebras 
$$
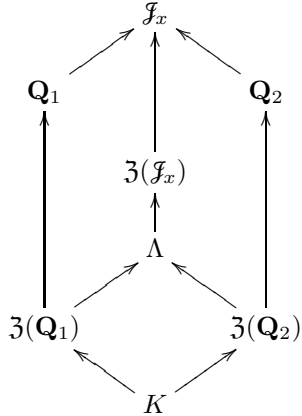
\xymatrix@R=15pt@C=10pt{&\eJac_x\\
\mathbf{Q}_1\ar[ur]&&\mathbf{Q}_2\ar[ul]\\
&\cent(\eJac_x)\ar[uu]\\
&\Lambda\ar[u]\\
\cent(\mathbf{Q}_1)\ar[uuu]\ar[ur]&& \cent(\mathbf{Q}_2)\ar[uuu]\ar[ul]\\
&\ar[ur]K\ar[ul]
}$$where $\mathbf{Q}_1$ and $\mathbf{Q}_2$ are the quantum planes
$$\mathbf{Q}_1:=\frac{k(\mm)\{\eps_0, \eps_1\}}{(\eps_0\eps_1-\zeta\eps_1\eps_0)}\quad\text{and}\quad \mathbf{Q}_2:=\frac{k(\mm)\{\eps_0, \eps_2\}}{(\eps_2\eps_0-\zeta\eps_0\eps_2)}$$inside $\eJac_x$. 

The point is that this diagram \emph{parametrises} cyclic algebras over $K$ and Brauer classes. Furthermore, this includes a \emph{non-commutative} family of central simple algebras. 

First of all, we have seen in (\ref{eq:quantumplane}) that the centre $\cent(\mathbf{Q}_K)$ of a quantum plane parametrises cyclic algebras by its very construction. The previous two propositions give information on this parametrisation by extending the centres $\cent(\mathbf{Q}_1)$ and $\cent(\mathbf{Q}_2)$ to the central algebra $\Lambda$. 

\begin{prop}\label{prop:Brauer1}Let $\mm$ be the maximal ideal 
$$\mm:=\mm(a,b,c)\in\azu_\Lambda(\eJac_x),\quad\text{with $a, b, c\in k(\mm)$ such that $abc\neq 0$}, $$and put $L:=k(\mm)[t]/(t^n-a)$. Then the following statements hold.
\begin{itemize}
	\item[(a)] There are at least two symbol algebras $(a,b)_{\zeta}$ and $(a, c)_{\zeta}$ in $J_\mm$. These are non-isomorphic unless $c/b\in\mathrm{Nm}(L)$. 
	\item[(b)] In addition:
\begin{itemize}
	\item[(i)] if $c/b\notin\mathrm{Nm}(L)$, then $J_\mm$ generates two classes in $\mathrm{Br}(k(\mm))[n]$;
	\item[(ii)] if $c/b\in\mathrm{Nm}(L)$, $J_\mm$ generates one class in $\mathrm{Br}(k(\mm))[n]$.
\end{itemize} 
	\item[(c)] Let $\Mm\in \Max(\eJac_x)$ be such that $\mm:=\Mm\cap \Lambda\in\azu_\Lambda(\eJac_x)$ and assume that $k(\Mm\cap \cent(\eJac_x))=k(\mm)$. Then there are two canonical central simple algebras
	\begin{equation}\label{eq:tensorcyclic}
	(a, b)_{\zeta, /k(\mm)}\underset{{k(\mm)}}{\otimes}(a, c)_{\zeta, /k(\mm)}\simeq (a, bc)_{\zeta, /k(\mm)}\in\mathrm{Br}(k(\mm))[n]
	\end{equation}and
		\begin{align}\label{eq:tensorcentral}
		\begin{split}
	(a, b)_{\zeta, /k(\mm)}\underset{{k(\mm)}}{\otimes}(a, c)_{\zeta, /k(\mm)}&\underset{{k(\mm)}}{\otimes}\eJac_x/\Mm\\
	&\simeq (a, bc)_{\zeta, /k(\mm)}\underset{{k(\mm)}}{\otimes}\eJac_x/\Mm\\
	&\in\mathrm{Br}(k(\mm))\Big[\mathrm{lcm}\big(n, \mathrm{per}(\eJac_x/\Mm)\big)\Big],
	\end{split}
	\end{align}where $\mathrm{per}(A)$ denotes the period (i.e., the order of the class of $A$ in $\mathrm{Br}(K)$) of the $K$-central simple algebra $A$. This implies that $\eJac_x$ parametrises cyclic algebras and Brauer classes in $\mathrm{Br}(k(\mm))$. 
\end{itemize}
\end{prop}
\begin{proof}
\begin{itemize}
	\item[(a)] We observe that there are two symbol algebras as subalgebras inside $J_\mm$. Namely,
$(a,b)_{\zeta}$ and $(a,c)_{\zeta^{-1}}$:
\begin{align*}
(a,b)_{\zeta}&= \frac{k(\mm)\{\eps_0,\eps_1\}}{\left(\eps_0^n-a,\eps_1^n-b,\,\,\eps_0\eps_1-\zeta\eps_1\eps_0\right)},\\
(a,c)_{\zeta^{-1}}&=\frac{k(\mm)\{\eps_0,\eps_2\}}{\left(\eps_0^n-a,\eps_2^n-c,\,\,\eps_2\eps_0-\zeta\eps_0\eps_2\right)}.
\end{align*}Since $(a,c)_{\zeta^{-1}}$ is the symbol algebra constructed as in (\ref{eq:symbol}), but with $$\ee t= \sigma^{-1}(t)\ee= \sigma^{n-1}(t)\ee,$$ we find that $(a,c)_{\zeta^{-1}}\simeq (a,c)_{\zeta}$. Indeed, if $\gcd(n,k)=1$, then $(a,c)_\sigma\simeq (a,c)_{\sigma^k}$, by a change of basis. Because $\gcd(n, n-1)=1$, the claim then follows. 

The reason for the addition of ``at least'' in the claim, is that, as we remarked, even if $\mm\in\azu_\Lambda(\eJac_x)$, there might be more than one fibre of $\eJac_x$ over $\mm$ depending on the splitting of $\mm$ inside $\cent(\eJac_x)$. 
	\item[(b)] The claim follows from the above discussion of symbol algebras above, and from the equivalence
$$A\simeq B\iff [A]=[B]\in\mathrm{Br}(F)\,\, \text{and}\,\, \dim_F(A)=\dim_F(B),$$for simple $F$-algebras $A$ and $B$. 
\item[(c)] Take a maximal ideal $\Mm$ in $\eJac_x$ and assume that $k(\Mm\cap \cent(\eJac_x))=k(\mm)$, with $\mm:=\Mm\cap \Lambda$. Then $\eJac_x/\Mm$ is a central simple algebra over $k(\mm)$. Let $k_1:=k(\mm\cap \cent(\mathbf{Q}_1))$ and $k_2:=k(\mm\cap\cent(\mathbf{Q}_2))$. Then $k_1, k_2\subseteq k(\mm)$ and we can extend the fibres to $k(\mm)$:
$$\big(\mathbf{Q}_1\underset{\cent(\mathbf{Q}_1)}{\otimes}k_1\big)\otimes_{k_1}k(\mm)=(a, b)_{\zeta,/k(\mm)}$$ and similarly with $\mathbf{Q}_2$ to get $(a, c)_{\zeta, /k(\mm)}$. Observe that the extension to $k(\mm)$ might split $(a, b)_{\zeta,/k(\mm)}$ or $(a, c)_{\zeta,/k(\mm)}$ (or both). It is well-known that $(a, b)\otimes (a,c)\simeq (a, bc)$ (e.g., \cite[Ex. 4.10(a)]{GilleSzamuely}) and that cyclic algebras are $n$-torsion in the Brauer group. Hence (\ref{eq:tensorcyclic}) follows. Since $\eJac_x/\Mm$ is central simple over $k(\mm)$ and has order $\mathrm{per}(\eJac_x/\Mm)$ the tensor product is $\mathrm{lcm}\big(n, \mathrm{per}(\eJac_x/\Mm)\big)$-torsion in $\mathrm{Br}(k(\mm))$. 
\end{itemize}The proof is complete. 
\end{proof}

Turning now to the case when $abc=0$, we have already found that the algebra $J_\mm$ is not simple and so $\mm\in\ram_\Lambda(\eJac_x)$.  Therefore we see that 
$$\{u_0=0\}\cup\{u_1=0\}\cup\{u_2=0\}\subseteq \ram_\Lambda(\eJac_x).$$ For instance, over the $\{u_0=0\}$-part of $\ram_\Lambda(\eJac_x)$, we find 
$$J_\mm=\eJac_x\otimes_{\Lambda}k(\mm)=\frac{\eJac_x}{(\eps_0^n,\eps_1^n-b,\eps_2^n-c)}, \quad bc\neq 0,\quad (\eps_0)\subseteq\mathrm{rad}(J_\mm),$$ since the radical is the largest nilpotent ideal and $\eps_0^n=0$. Hence, if $x\neq 0$, 
$$J_\mm\big/(\eps_0)=\frac{k(\mm)\{\eps_1,\eps_2\}}{(\eps_1^n-b,\,\,\eps_2^n-c,\,\,\eps_2\eps_1-\zeta^{2}\eps_1\eps_2-(1-\zeta^{2})x)},$$ and if $x=0$ we get the cyclic algebra $(b, c)_{\zeta^2}$. 

If $\zeta^{2}\neq 1$ and $x\neq 0$, we can change basis $\eps_1\mapsto (1-\zeta^{2})x\eps_1$ and see that $J_\mm/(\eps_0)$ is isomorphic to the quotient of the first quantum Weyl algebra
$$A_{1}(\zeta^2)_{/k(\mm)}=\frac{k(\mm)\{\vv, \boldsymbol{w}\}}{(\vv \boldsymbol{w}-\zeta^{2}\boldsymbol{w}\vv-1)}, \qquad \cent\big(A_{1}(\zeta^2)_{/k(\mm)}\big)=k(\mm)[\vv^l, \boldsymbol{w}^l],$$ by the central maximal ideal $(\vv^l-b, \boldsymbol{w}^l-c)$, where $l$ is the least integer such that $2l$ is a multiple of $n$. It is known (see e.g., \cite[Theorem 6.2]{Irving2}) that localising $A_{1}(\zeta^2)_{/k(\mm)}$ at a certain central element, $\omega$, gives an Azumaya algebra. Therefore, the ramification locus in $\cent\big(A_{1}(\zeta^2)_{/k(\mm)}\big)$ is the zero set of this element. In fact, this $\omega$ gives the hyperbola $\vv\boldsymbol{w} = (1-\zeta^2)^{-l}$ inside $\cent\big(A_{1}(\zeta^2)_{/k(\mm)}\big)$ and $(b,c)\in\ram\big(A_{1}(\zeta^2)_{/k(\mm)}\big)$ if and only if $bc=(1-\zeta^2)^{-l}$. 

Clearly, if $\zeta^{2}=1$, the algebra $J_\mm/(\eps_0)$ is a (commutative) zero-dimensional subscheme embedded into $\mathbb{A}^2$. 

The cases $b=0$ and $c=0$, both of which are interchangeable, are more subtle (and to be honest, I don't quite understand everything here myself). First recall that for any artinian algebra $A$, $A$ and its semisimple quotient $A/\mathrm{rad}(A)$ have the same simple modules. Let $\mm=(\eps_0^n -a, \eps_1^n, \eps_2^n-c)$ and look at $J_\mm=\eJac_x/\mm$. The radical of $J_\mm$ is $(\eps_1)$ so $J_\mm$ and $J_\mm/(\eps_1)$ have the same simple modules and we have a surjection
$$J_\mm \twoheadrightarrow J_\mm/(\eps_1)=\frac{k(\mm)\{\eps_0, \eps_2\}}{\Big(\eps_0^n-a,\,\, \eps_2^n-c,\,\,\eps_0\eps_2-\zeta\eps_2\eps_0,\,\, x\eps_0+(1-\zeta^2)x\Big)}.$$Assuming first that $x\neq 0$, this clearly means that $J_\mm/(\eps_1)$ only have the one-dimensional simple module defined by
\begin{equation}\label{eq:1dmodule1}
M=k(\mm)\cdot\uu, \quad \eps_0\cdot\uu=-(1-\zeta^2)\uu,  \quad \eps_2\cdot\uu =0.
\end{equation} Hence the only simple $J_\mm$-module is also $M$, with $\eps_1$ acting as $b=0$ on $\uu$. Observe that this is independent on $a$ and $c$: regardless of the values of $a$ and $c$, the action of $J_\mm$ on $M$ is given as above. This also implies that $M$ is a simple one-dimensional $\eJac_x$-module via the composition $\eJac_x\twoheadrightarrow J_\mm \twoheadrightarrow J_\mm/(\eps_1)$. 

If $x=0$, we get the symbol algebra $(a, c)_\zeta$, which is central simple. 

The point of the above discussion is that even though there are several maximal ideals (simple modules) of $\eJac_x$ above $\mm\in \ram_R(\eJac_x)$, the algebraic fibre $J_\mm$ has only one simple module $M$, which is also then a simple module of $\eJac_x$. 

We summarise the discussion above in the following proposition.
\begin{prop}\label{prop:ram}
Let $abc=0$. Then the following can occur:
\begin{itemize}
	\item[(a)] Over the plane $u_0=0$ the fibre $J_\mm$ is a non-commutative deformation of a central quotient of the first quantum Weyl algebra $A_1(\zeta^2)_{/k(\mm)}$, 
	$$\frac{\eJac_x}{(\eps_0^n,\eps_1^n-b,\eps_2^n-c)}.$$
	whose semisimplification is $A_1(\zeta^2_p)/(\vv^p-b, \boldsymbol{w}^p-c)$ when $x\neq 0$ and $(a, c)_\zeta$ when $x=0$. 
	\item[(b)] Over the plane $u_1=0$, $J_\mm$ is a non-commutative deformation of 
	\begin{itemize}
		\item[(i)] the fat point $k(\mm)[\eps_0]/\big(\eps_0^n-(-1)^n(1-\zeta^2)^n\big)$ if $x\neq 0$, $\zeta^2\neq 1$;
		\item[(ii)] the cyclic algebra $(a, c)_\zeta$ if $x=0$, $\zeta^2\neq 1$;
		\item[(iii)] the algebra
		$$\frac{k(\mm)\{\eps_0, \eps_2\}}{\Big(\eps_0^2-a,\,\, \eps_2^2-c,\,\,\eps_0\eps_2+\eps_2\eps_0\Big)}$$if $x=0$, $\zeta^2=1$, and
		\item[(iv)] the algebra $k(\mm)[\eps_2]/(\eps_2^n-c)$ if $x\neq 0$, $\zeta^2=1$. 
	\end{itemize}The cases $u_1=0$ and $u_2=0$ are completely symmetric. 
\end{itemize}When $x\neq 0$, the fibre $J_\mm$ has the simple module given by (\ref{eq:1dmodule1}), which is also a simple module for $\eJac_x$. In the case $x=0$, the simple module is the simple module of $(a, c)_\zeta$. 
\end{prop}
We end this section with an example meant to inspire the reader to go where I dare not at this point. 
\begin{example}The following situation should be worth pondering in some detail in view of the interest of Brauer groups and division algebras over surfaces (see for instance \cite{SaltmanDivsionSurface}). 

Let $j_\cent: S_\cent \to \Spec(\cent(\eJac_x))$ be a closed immersion, where $S_\cent$ is a two-dimensional normal subscheme of $\Spec(\cent(\eJac_x))$. The induced morphism over $\mathbb{A}^3_{(n)}=\Spec(\Lambda)$ is denoted $j: S_\Lambda \to \mathbb{A}^3_{(n)}$. 

The sheaf $\eJac_x$ over $\Spec(\cent(\eJac_x))$ and $\mathbb{A}^3_{(n)}$ has pull-backs $j_\cent^\ast\eJac_x$ and $j_\Lambda^\ast\eJac_x$ over $S_\cent$ and $S_\Lambda$, respectively. The algebras $j_\cent^\ast\eJac_x$ and $j_\Lambda^\ast\eJac_x$ are maximal orders in $j_\cent^\ast\eJac_x\otimes_{S_\cent} K(S_\cent)$ and $j_\Lambda^\ast\eJac_x\otimes_{S_\Lambda}K(S_\Lambda)$. 

Now, propositions \ref{prop:Brauer1} and \ref{prop:ram} seem to give interesting information concerning the Brauer groups at the different points of the  surfaces $S_\cent$ and $S_\Lambda$. This certainly includes the generic points. The reader is invited to examine this in more detail. 
\end{example}

\section{Rational points on $\Xscr_{\eJac_x(r)}$}\label{sec:rational_points}
From now on, unless explicitly stated otherwise, we assume that $B$ is a field $K$ such that $K\supset\mu_{n}$. We begin by looking at the ``commutative points''. 

\subsection{The locus of one-dimensional points and its infinitesimal structure}\label{sec:point_locus_tangents}
Let $K\subseteq L$ be a field extension. A one-dimensional representation $\Pp=L\cdot\uu$ must come from a maximal ideal $\mathfrak{M}$ in $\eJac_x(r)$ and $\Pp=\eJac_x(r)/\mathfrak{M}$. Since $L$ is a field, the ideal $\mathfrak{M}$ must include the ideal generated by the commutators of $\eps_0$, $\eps_1$ and $\eps_2$. Assume $\zeta^{2r}\neq 1$. Then the relations reduce to 
\begin{align*}
(1-\zeta^r)\eps_0\eps_1&=0\\
(1-\zeta^r)\eps_0\eps_2&=0\\
(1-\zeta^{2r})\eps_1\eps_2&=x\eps_0+x(1-\zeta^{2r}).
\end{align*}This implies that $x(\eps_0+(1-\zeta^{2r}))=0$ so, if $x\neq 0$, we have that $(\zeta^{2r}-1,0,0)$ is a one-dimensional module. In addition, the conic (hyperbola) $C_x$ given by $\eps_1\eps_2=x$ in the plane $\eps_0=0$ consists entirely of commutative points (i.e., one-dimensional modules). When $x=0$, we find that the one-dimensional modules lie on the union of the coordinate axes. 

\begin{remark}\label{rem:zeta2=1}When $\zeta^{2r}=1$ we may assume that $\zeta^r=-1$. The relations
$$\eps_0\eps_1+\eps_1\eps_0=0,\quad \eps_2\eps_0+\eps_0\eps_2=0,\quad \eps_2\eps_1-\eps_1\eps_2=x\eps_0$$imply that
$$2\eps_0\eps_1=0,\quad 2\eps_0\eps_2=0,\quad x\eps_0=0.$$Consider first when $\mathrm{char}(K)\neq 2$. When $x\neq 0$ we get that the plane $\{\eps_0=0\}$ consists entirely of one-dimensional points. On the other hand, when $x=0$, we get the union of the coordinate planes. 
Now, if $\mathrm{char}(K)=2$, we get, when $x\neq 0$, the plane $\{\eps_0=0\}$ and, if $x=0$, the whole $\mathbb{A}^3_{/K}$. 
\end{remark}

Put $u_1:=\eps_1^n$ and $u_2:=\eps_2^n$. The restriction, $\cent(C_x)$, of $C_x$ to the centre is then given by the equation $u_1u_2=x^n$ and a point $$\mathfrak{M}=(\eps_0, \eps_1-b, \eps_2-c),\quad b,c\in L,$$ on $C_x$ restricts to the point $$\mathfrak{m}=(u_0, u_1-b', u_2-c'),\quad\text{with}\,\,\, b'=b^n, \,\, c'=c^n\in L,\,\,\text{such that}\,\, \, b'c'=x^n.$$Observe that 
$$\mathfrak{M}_{ij}=(\eps_0,\,\, \eps_1-\zeta^i b,\,\, \eps_2-\zeta^j c),\quad 1\leq i,j\leq n-1,$$ also maps to $\mathfrak{m}$ under $\phi$, so $\mathfrak{M}_{ij}\in\phi^{-1}(\mathfrak{m})$. This then means  that $\cent(C_x)$ lies in the ramification locus. 

Be sure to note that the divisors $Y_{\aa}$, for all $\aa$ with $a_0\neq a_1$, intersect the $\{\eps_0=0\}$-part of the ramification locus along $C_x$. 

We will now compute the tangents between the one-dimensional modules. From now on we will for simplicity only consider the case $r=1$ and leave it to the reader to insert $r$ at the appropriate places. Therefore, put $\eJac_x:=\eJac_x(1)$. 

\begin{remark}\label{rem:ExtHoch}The $\Ext$-computations below are done in the Hochschild complex. Recall that for an $K$-algebra $R$ we have the following isomorphisms
$$\Ext_R^1(M_1,M_2)\simeq \mathrm{HH}^1(R,\Hom_K(M_1,M_2))\simeq \Der_K(R,\Hom_K(M_1,M_2))/\mathrm{Ad},$$where $\mathrm{HH}^1$ is the Hochschild cohomology functor and $\mathrm{Ad}$ is the subgroup of $\Der(R,\Hom(M_1,M_2))$ of inner derivations. 
\end{remark}

Put $\Pp_1:=\eJac_x/\mathfrak{M}_1$ and $\Pp_2:=\eJac_x/\mathfrak{M}_2$, with $\mathfrak{M}_1,\mathfrak{M}_2$ lying over $ \mm\in\Max(\cent(\eJac_x))$, are both one-dimensional over $L$. It is easy to see that $\eps_0$ must act as zero on any one-dimensional module. Therefore, we put 
$$\Pp(b,c):=\Pp_1 = L\cdot\uu, \quad\text{with}\quad \eps_1\cdot\uu=b\uu,\quad \eps_2\cdot\uu=c\uu, \quad b,c\in L,$$and similarly $\Pp(e, f):=\Pp_2$. Furthermore, put
$$\delta(\eps_0)\uu=\Delta_{0}\vv,\quad \delta(\eps_1)\uu=\Delta_{1}\vv,\quad \delta(\eps_2)\uu=\Delta_{2}\vv,\quad \Delta_{0},\Delta_{1},\Delta_{2}\in L,$$ where we chose $\vv$ as basis for $\Pp(e,f)$. 

Assume first that $x\neq 0$. The relations of $\eJac_x$ must map to zero in the group $\Hom_L\!\!\big(\Pp(b,c),\Pp(e,f)\big)$. This gives restrictions on the set $\{b,c,e, f, \Delta_0,\Delta_1, \Delta_2\}$. A small computation shows that these restrictions are:
\begin{itemize}
	\item $\Delta_0=0$. 
	\item $b=\zeta e$, and if $f\neq \zeta c$ we get one free parameter; if not, we find that $\Delta_{1}$ and $\Delta_{2}$ are both free. 
	\item $b=\zeta^2 e$: here we get that $(b-\zeta e)(f-\zeta^2 c)\Delta_{1}=0$, so if $f=\zeta^2 c$, then $\Delta_{1}$ and $\Delta_{2}$ are both free, otherwise only $\Delta_{2}$ is free.
	\item By symmetry we can switch $(b,e)$ and $(c,f)$ and get the same result. 
\end{itemize}
As for the inner derivations 
$$\theta\in \Hom_L\!\!\big(\Pp(b,c),\Pp(e,f)\big)\quad \theta(\uu)=s\vv,\,\, s\in L,$$
we get 
\begin{align*}
\mathrm{ad}_{\theta}(\eps_1)(\uu)&=\theta \eps_1\uu-\eps_1\theta\uu=(b-e)s\vv\\
\mathrm{ad}_{\theta}(\eps_2)(\uu)&=\theta \eps_2\uu-\eps_2\theta\uu=(c-f)s\vv
\end{align*}implying that $\dim(\mathrm{Ad})=1$ unless $b=e$ and $c=f$, in which case $\dim(\mathrm{Ad})=0$.

The relevant computations yield the following propositions. 

\begin{prop}When $x\neq 0$ and $\zeta^{2}\neq 1$, the $\Ext^1$-groups are
$$
	\Ext_A^1\!\!\big(\Pp(b,c),\Pp(e,f)\big)=\begin{cases}
	L, &\text{if $a=\zeta e$ and $f=\zeta c$}\\
	L &\text{if $a=\zeta^2 e$ and $f=\zeta^2 c$}\\
	L &\text{if $a=e$ and $c=f$} \\
	0, &\text{otherwise.}
	\end{cases}
$$	
\end{prop} 

Similar reasoning yields the case $x=0$:
\begin{prop}When $x=0$ and $\zeta^2\neq 1$, the $\Ext^1$-groups are
$$
\Ext_A^1\!\!\big(\Pp(b,c),\Pp(e,f)\big)=\begin{cases}
L & \text{if}\quad  b=\zeta^2 e,\,\, f=\zeta^2 c\\
L^2 & \text{if}\quad b=c=e=f=0\\
0 & \text{otherwise.}
\end{cases}
$$
\end{prop}
The cases when $\zeta^2=1$ are more complicated since we then need to handle the case of characteristic two separately. Analogous computations as above gives the following propositions. 
\begin{prop}When $x\neq 0$ and $\zeta^{2}= 1$,
$\Ext_A^1\!\!\big(\Pp(b, c), \Pp(b, c)\big)=L$, and zero in all other cases, independent on characteristic. 	
\end{prop} 
\begin{prop}When $x=0$ and $\zeta^2= 1$, the $\Ext^1$-groups are
\begin{itemize}
	\item[(i)] $\mathrm{char}(K)\neq 2$: here $\Ext_A^1\!\!\big(\Pp(b,c),\Pp(\pm b,\pm c)\big)=L$ and zero otherwise;
	\item[(ii)] $\mathrm{char}(K)=2$: here $\Ext_A^1\!\!\big(\Pp(b,c),\Pp(b,c)\big)=L^3$ and zero otherwise. 
\end{itemize}
\end{prop}

The above propositions determine the tangent space $T_{\{\Pp(b,c), \Pp(e,f)\}}$. 
 
\subsection{Families of rational points}

We will follow the construction given by D. Jordan in \cite{Jordan} (with a few modifications so as to apply in our case) to identify some (but not necessarily all) finite-dimensional simple modules of $\eJac_x(r)$. Jordan's construction gives a moduli for all simple modules over an algebraically closed field. 

Put $\yy:=x(1-\zeta^r)^{-1}\eps_0+x$ and $$\ww:=\eps_2\eps_1-\yy=\eps_2\eps_1-x(1-\zeta^r)^{-1}\eps_0-x.$$Recall from the proof of proposition \ref{prop:isoJ} that we had an automorphism $\tau$ on $\eJac_x$, defined on $\eps_0$ as $\tau(\eps_0):=\zeta^{-r}\eps_0$. The order of $\tau$ is $\mathrm{lcm}(r,n)$. 

We note that  $$\eps_0\ww=\ww\eps_0,\quad \eps_1 \ww=\zeta^{-2r}\ww\eps_1,\quad \eps_2\ww=\zeta^{2r}\ww\eps_2$$ and $\ww=\zeta^{2r}(\eps_1\eps_2-\tau(\yy))$. We extend $\tau$ to the commutative $K$-algebra $K[\eps_0][\ww]$ by putting $\tau(\ww)=\zeta^{-2r}\ww$. Observe that this is \emph{not} consistent with the definition of $\ww$ with respect to the $\tau$ as given in proposition \ref{prop:isoJ}. For instance, $\tau(\eps_2)$ is not defined. However, this is not a problem as we will see.

In addition, let $\mm=(f)\in\Max(K[\eps_0])$, for $f$ irreducible. Recall that $x\in K$. If $x=0$, clearly $\yy=0$.


The identity
\begin{equation}\label{eq:Jordan_id1}
\eps_2\eps_1^i=\big(\yy-\zeta^{2ri}\tau^i(\yy)\big)\eps_1^{i-1}+\zeta^{2ir}\eps_1^i\eps_2
\end{equation}or equivalently in our case, 
\begin{equation}\label{eq:Jordan_id2}
\eps_2\eps_1^i=x\big([i]_{\zeta^r}\eps_0+(1-\zeta^{2ri})\big)\eps_1^{i-1}+\zeta^{2ir}\eps_1^i\eps_2,
\end{equation}follows by a simple induction argument. 

\subsection{Torsion points}
Suppose that $d\in\Z_{>0}$ is minimal with the property that 
\begin{equation}\label{eq:torsion_min}
\yy-\zeta^{2rd}\tau^d(\yy)=x\big([d\,]_{\zeta^r}\eps_0+(1-\zeta^{2rd})\big)\in\mm.\end{equation} Then 
$$T(\mm):=\frac{\eJac_x(r)}{\eJac_x(r)(\mm,\eps_2,\eps_1^d)}$$ is a simple $\eps_{1}$- \emph{and} $\eps_2$-torsion module of dimension $d$. 

Explicitly, let $\mm=(f)$ be a maximal ideal in $K[\eps_0]$. However, because of (\ref{eq:torsion_min}), we see that 
$$x\big([d\,]_{\zeta^r}\eps_0+(1-\zeta^{2rd})\big)=x[d\,]_{\zeta^r}\big(\eps_0+(1-\zeta)(1+\zeta^{rd})\big)$$must be a factor in $f$ and so $\mm=(\eps_0-a)$, with $a=-(1-\zeta)(1+\zeta^{rd})$. 

Put $v_i:=\eps_1^i+I$, with $I:=\eJac_x(r)(\mm,\eps_2)$. Then,  
\begin{align}\label{eq:torsion_0_1}
	\eps_{0}\cdot v_i &=\eps_0\eps_1^i+I= \zeta^{ri} a\eps_1^i+I=\zeta^{ri} av_i,\quad\text{and}\quad \eps_1\cdot v_i=v_{i+1},
\end{align}where the second equality follows since $\eps_0\eps_{1}+I=\zeta^r\eps_1\eps_0+I=\zeta^r\eps_1 a+I$. Also,
\begin{align}\label{eq:torsion_n-1}
\begin{split}
	\eps_2\cdot v_i&=\eps_2\eps_1^i+I\overset{(\ref{eq:Jordan_id2})}{=}x\big([i]_{\zeta^r}\eps_0+(1-\zeta^{2ri})\big)\eps_1^{i-1}+I\\
	&=x\big(a[i]_{\zeta^r}\zeta^{r(i-1)}+(1-\zeta^{2ri})\big)v_{i-1}=\psi_i^\mathrm{t}(a)v_{i-1},
	\end{split}
\end{align}with $$\psi_i^\mathrm{t}(a):=x\big(a[i]_{\zeta^r}\zeta^{r(i-1)}+(1-\zeta^{2ri})\big).$$ Therefore,
$$T(\mm)=\bigoplus_{i=0}^{d-1}K\cdot v_i,$$with the actions (\ref{eq:torsion_0_1}) and (\ref{eq:torsion_n-1}). We see that when $x=0$, we get $\eps_2\cdot v_i=0$, for all $i\neq 0$. Observe that the torsion points are defined already over $K$.

Notice that 
$$\rho(\eps_0)^d=a^d I,\quad\text{and}\quad \rho(\eps_1)^d=\rho(\eps_2)^d=\mathbf{0}.$$When $x=0$ we get $\rho(\eps_2)=\mathbf{0}$. 

The above discussion proves the following proposition.
\begin{prop} Given $(d,r)\in\Z_{>0}^2$, the element $a$ is uniquely determined and so for each distinct choice $(d,r)$ there is exactly one torsion module up to isomorphism. In other words, for each pair $(d,r)\in\Z_{>0}^2$ there is a \emph{unique} $\eps_1\eps_2$-torsion point in $\Xscr_{\eJac_x}(K)$ and all torsion points are $K$-rational. 
\end{prop}

\begin{example}When $n=d=3$, $K=\Q(\zeta_3)$ and $x=1$, we get a family $T(\mm)=T(a)$ of modules, parametrised by $\mm=(\eps_0-a)$, $a\in K$. On matrix form $T(a)$ is given as (with $r=1$),
\begin{align*}
\rho_a(\eps_0)&=\begin{pmatrix}
a & 0 & 0\\
0 &\zeta_3 a & 0\\
0 & 0 & \zeta_3^2 a
\end{pmatrix},\quad
 \rho_a(\eps_1)=\begin{pmatrix}
	0 & 0 & 0\\
	1 & 0 & 0\\
	0 & 1 & 0
\end{pmatrix}, \\
\rho_a(\eps_2)&=\begin{pmatrix}
0 & a+\zeta_3+2 & 0\\
0 & 0 & 1-a-\zeta_3\\
0 & 0 & 0
\end{pmatrix}
\end{align*}
Note that the module lies over the point $\left(a^3,0,0\right)\in\Specm(\cent(\eJac_x))(K)$. 
\end{example}
\subsection{Torsion-free points} 
To determine the $\eps_1$- and $\eps_2$-torsion-free modules we proceed as follows. 

Take $\mm=(f)\in\Max(K[\eps_0])$ such that $\tau^d(\mm)=\mm$, with $d\in\Z_{>0}$ minimal with this property. Let $L$ be the splitting field of $f$ and $a\in L$ such that $f(a)=0$. Put 
$$\Mm_b:=K[\eps_0][\ww](\mm, \ww-b)\subseteq (K[\eps_0]\otimes_K L)[\ww]=L[\eps_0][\ww],\quad b\in L.$$Observe that $\eps_0 \ww=\ww\eps_0$ so $K[\eps_0][\ww]$ is a commutative subalgebra of $\eJac_x(r)$. Let $f=\sum_{i=0}^r f_i\eps_0^i$ be a generator for $\mm$. Then, recalling that $\tau(\eps_0)=\zeta^{-r}\eps_0$, we get
$$\tau^d(f)=\sum_{i=0}^r f_i\tau^d(\eps_0)^i=\sum_{i=0}^r f_i \zeta^{-rdi}\eps_0^i,$$from which it follows that 
$$\tau^d(\Mm_b)=\left(\sum_{i=0}^r f_i \zeta^{-rdi}\eps_0^i, \zeta^{-2rd}\ww-b\right).$$ Now put
$$J:=\eJac_x(r)(\Mm_b, \eps_1^s-c), \quad c\in L,$$and $v_i:=\eps_1^i+J$. Since $\ww=\eps_2\eps_1-\yy$, we find 
\begin{align*}
	\eps_2\eps_1^i=(\eps_2\eps_1)\eps_1^{i-1}
	=(\ww+\yy)\eps_1^{i-1}.
\end{align*}Furthermore,
$\ww\eps_1+J=\zeta^{2r}\eps_1\ww+J=\zeta^{2r}b\eps_1+J$ and 
\begin{align*}
\yy\eps_1+J&=\big(x(1-\zeta^r)^{-1}\eps_0+x\big)\eps_1+J=\eps_1\big(x(1-\zeta^r)^{-1}\zeta^r\eps_0+x\big)+J\\
&=\big(x(1-\zeta^r)^{-1}\zeta^ra+x\big)\eps_1+J.
\end{align*}This implies that 
$$\eps_2\eps_1^i+J=\Big(\zeta^{2r(i-1)}b+x\big(a\zeta^{r(i-1)}(1-\zeta^r)^{-1}+1\big)\Big)\eps_1^{i-1}+J,$$hence
\begin{equation}\label{eq:torsionfree_e_(n-1)}
\eps_2\cdot v_i=\psi_i^\mathrm{tf}(a,b) v_{i-1},
\end{equation}
where 
$$\psi_i^\mathrm{tf}(a,b):=\Big(\zeta^{2r(i-1)}b+x\big(a\zeta^{r(i-1)}(1-\zeta^r)^{-1}+1\big)\Big).$$Clearly, $\psi_i^{\mathrm{tf}}$ is a function $\psi_i^{\mathrm{tf}}:L\times L\to L$. 

Modulo $J$ we have that $\eps_1^s=c$, so $\eps_1^sc^{-1}=1=\eps_1^0$. From this observation follows
$$\eps_2\eps_1^0+J=\eps_2c^{-1}\eps_1^s=c^{-1}\big(\ww+\yy)\eps_1^{s-1}+J,$$so
\begin{equation}\label{eq:torsionfree_e_(n-1)_v0}
\eps_2\cdot v_0=c^{-1}\psi_s^\mathrm{tf}(a,b) v_{s-1}.
\end{equation}
Clearly 
\begin{equation}\label{eq:torsionfree_e_1}
\begin{split}
\eps_1\cdot v_i=\begin{cases}
	v_{i+1}, &\text{if } 0\leq i<s-1\\
	c v_0, &\text{if } i=s-1
	\end{cases}
\end{split}
\end{equation}and
\begin{equation}\label{eq:torsionfree_e_0}
	\eps_0\cdot v_i=\zeta^{ri}a v_i, \quad 0\leq i\leq s.
\end{equation}The case $x=0$ gives 
\begin{equation}\label{eq:torsionfree_e_(n-1)_x=0}
\eps_2\cdot v_i=
\begin{cases}
\zeta^{2r(i-1)}bv_{i-1}, & i\neq 0\\
c^{-1}\zeta^{2r(s-1)}bv_{s-1}, & i=0.
\end{cases}
\end{equation}
We have that 
$$\rho(\eps_0)^s=a^s I,\quad \rho(\eps_1)^s=c, \quad\text{and}\quad\rho(\eps_2)^s=\prod\limits_{i=1}^{s} \psi^{\mathrm{tf}}_i(a,b).$$When $x=0$ we get 
$$\rho(\eps_2)=
	\begin{pmatrix}
		0 & b & 0 & 0 & \cdots & 0\\
		0 & 0 & \zeta_n^{2r}b & 0 & \cdots & 0\\
		\vdots & \vdots & \ddots & \ddots & \ddots&0\\
		0 & 0 & 0 & \cdots & \cdots& \zeta_n^{2r(s-2)}b\\
		c^{-1}\zeta_n^{2r(s-1)}b & 0 & 0 & \cdots & \cdots & 0
	\end{pmatrix}$$and the other two operators are the same. Observe that this implies that, when $x\neq0$, the corresponding maximal ideal lies over the point
	$$\left(a^s,c,\prod\limits_{i=1}^{s} \psi^{\mathrm{tf}}_i(a,b)\right)\in\Specm\big(\cent(\eJac_x(r))\big).$$

There are two distinct cases to consider, giving simple $\eps_1$-torsion-free $\eJac_x(r)$-modules:
\begin{itemize}
	\item[(1)] $b=0$. Here $\tau^d(\Mm_0)=\Mm_0$ and we define
	$$M(\mm, c):=\frac{\eJac_x(r)}{\eJac_x(r)(\Mm_0, \eps_1^d-c)}=\bigoplus_{i=0}^{d-1}L\cdot v_i.$$
	\item[(2)] $b\neq0$. Here $\Mm_b$ is periodic since $\zeta^{2r}$ is a root of unity. The order of $\Mm_b$ is $s=\mathrm{lcm}(d,l)$, with $l$ the least integer such that $(\zeta^{2r})^l=1$. Define
	$$M(\mm, b,c):=\frac{\eJac_x(r)}{\eJac_x(r)(\Mm_b, \eps_1^s-c)}=\bigoplus_{i=0}^{s-1}L\cdot v_i.$$
\end{itemize}
In both cases the actions are given by formulas (\ref{eq:torsionfree_e_(n-1)})--(\ref{eq:torsionfree_e_0}). 

There is an obvious $\eps_1$-$\eps_2$-symmetry which gives that the same construction yields the $\eps_2$-torsion-free modules $M'(\mm,c)$ and 
$M'(\mm,b,c)$, as well. 

When $K=K^{\mathrm{al}}$, the above construction is a complete classification of simple finite-dimensional modules of $\eJac_x(r)$ up to isomorphism. When $K$ is not algebraically closed there are finite-dimensional simple modules not isomorphic to one in the above families (in other words, a module, simple over $K$, might reduce over $K^\mathrm{alg}$; there are ``more'' simple modules over $K^\mathrm{al}$ than over $K$). Studying how modules degenerate into families is actually quite subtle and should definitely be studied further.

We have already discussed the torsion-case above, so we now focus on the torsion-free situation. First,

\begin{example}
	The commutative points in section \ref{sec:point_locus_tangents} is gotten with $a=0$ and $d=1$. In other words, $\eps_0$ is a factor in $f$.
\end{example}

\begin{example}We still look at $n=d=3$, $r=1$, $x=1$, $\mm=(\eps_0-a)$, $a\in L$ and $K=\Q(\zeta_3)$. We find 
\begin{align*}
\psi^\mathrm{tf}_1(a,b)&=b+a(1-\zeta_3)^{-1}+1=\left(\frac{1}{3} \zeta_{3} + \frac{2}{3}\right) a + b + 1,\\
\psi_2^{\mathrm{tf}}(a,b)&=\zeta_3^2 b+a\zeta_3(1-\zeta_3)^{-1}+1=\left(\frac{1}{3} \zeta_{3} - \frac{1}{3}\right) a + \left(-\zeta_{3} - 1\right) b + 1,\\
\psi_3^\mathrm{tf}(a,b)&=c^{-1}\big(\zeta_3b+a\zeta_3^2(1-\zeta_3)^{-1}+1\big)=\left(-\frac{2}{3} \zeta_{3} - \frac{1}{3}\right) c^{-1}a + \zeta_{3} c^{-1}b + c^{-1}
\end{align*}
and matrices
\begin{align*}
\rho(\eps_0)&=\begin{pmatrix}
a & 0 & 0\\
0 &\zeta_3 a & 0\\
0 & 0 & \zeta_3^2 a
\end{pmatrix},\quad
 \rho(\eps_1)=\begin{pmatrix}
	0 & 0 & c\\
	1 & 0 & 0\\
	0 & 1 & 0
\end{pmatrix}, \\
\rho(\eps_2)&=\begin{pmatrix}
0 & \psi_1^\mathrm{tf}(a,b) & 0\\
0 & 0 & \psi_2^\mathrm{tf}(a,b)\\
c^{-1}\psi_3^\mathrm{tf}(a,b) & 0 & 0
\end{pmatrix}
\end{align*}The module $\rho$ lies over the point 
$$\left(a^3,c, \psi_1^\mathrm{tf}(a,b)\psi_2^\mathrm{tf}(a,b)\psi_3^\mathrm{tf}(a,b)\right)\in\Specm(\cent(\eJac_x))(L).$$Explicitly, the $z$-coordinate is
\begin{align*}&\Big(-\frac{1}{9} \zeta_{3} + \frac{1}{9}\Big) a^{3}c^{-1} + \left(-\zeta_{3} - 1\right) b^{3}c^{-1} + \left(2 \zeta_{3} + 1\right) a bc^{-1} - (\zeta_{3} + 1)c^{-1}.
\end{align*}
\end{example}
\begin{remark}\label{rem:modules} We make two remarks here.
\begin{itemize}
	\item[(i)] Fix $t_0,t_1,t_2\in L$ and $\mm_{t}=(u_0-t_0,u_1-t_1,u_2-t_2)\in\Max(\cent(\eJac_x))$. Then the extension of $\mm_{t}$ to $\eJac_x$ is
	$$\Mm_{t}=(\eps_0^3-t_0,\eps_1^3-t_1,\eps_2^3-t_2).$$A necessary condition for the associated module to be isomorphic to one in the above family is that $\sqrt[3]{t_0}\in L$. The same remark can be given for $\eps_2$. But notice that $\eps_1$ is already given on the correct form. Switching to the $\eps_2$-torsion-free versions, we get that $\eps_2$ is on the correct form. 
	
	However, if $\sqrt[3]{t_0}\notin L$ then $\rho_t$ is not isomorphic over $L$ to a module in the family. To get an isomorphism we need to extend to $L(\sqrt[3]{t_0})$. The same applies to $\eps_2$ and $t_2$. Therefore, over the field $K_{(2)}:=L[L^{1/3}]$, every module is isomorphic to a (unique) module in the family. 	The same remarks apply to the general case of arbitrary $n$. 
	\item[(ii)] Another thing worth remarking upon, is that in the construction above, we started from a commutative algebra and used induction to bigger rings and constructing modules along these inductions. In this process we have not used the possibility that some power of $\eps_0$ is torsion, without $\eps_0$ being torsion. The arguments in \cite{Jordan} show that any module, $\eps_0$-torsion or not, is isomorphic to one in the above families over $K^{\mathrm{al}}$ (clearly, it must be torsion with respect to \emph{some} element, though). However, this might not be true over fields which are not algebraically closed. 
\end{itemize}
\end{remark}

\begin{example}\label{ex:higher_Ext}We continue the above example and assume furthermore that $x\neq 0$. 
Let us look at the line $\ell$ in $\Specm(\cent(\eJac_x))$ parametrised by 
$$\cc(t):=(t, t+\zeta_3,t+9),\quad t\in L, $$ and the family 
$$M_t:=\frac{\eJac_x}{\eJac_x\mm_t},\qquad \mm_t:=\big(\eps_0^3-t,\,\, \eps_1^3-t-\zeta_3,\,\, \eps_2^3-t-9\big)$$ of (left) $\eJac_x$-modules. 

Clearly, $\ell$ intersects the ramification locus (not necessarily uniquely) in the point $P:=(0,\zeta_3,9)$. Let $M_0$ be a $\eJac_x$-module lying over $P$. For instance, $M_0=M(\mm_0,\zeta_3, 2)$, with $\mm_0=(u_0)$, is such a module since the contraction to $\cent(\eJac_x)$ is $(0,c,b^3+1)=(0,\zeta_3,9)$. 

However, $M_{i}:=M(\mm_0,\zeta_3,\zeta_3^i\cdot 2)$, $i=1,2$, are also two modules contracting to $P$. Hence the arithmetic fibre over $P$ is 

$$\Phi(P)=\{M_0, M_{1}, M_{2}\}.$$ Notice that these modules are non-isomorphic over $K^{\mathrm{al}}$ since they are distinct points in the family. Therefore they are non-isomorphic also over $K$. 

The geometric fibre is
$$\Xscr_{\eJac_x/\langle\mm_0\rangle}=\Mod(\eJac_x/\langle\mm_0\rangle).$$ (By $\langle\mm_0\rangle$ we mean the 2-sided ideal generated by $\mm_0$.) The support of $\Xscr_{\eJac_x/\langle\mm_0\rangle}$ is the point corresponding to the symbol algebra $(\zeta_3, 9)_{\zeta_3}$ (which is simple and so has only one simple module). 

After a, not quite trivial, computation one finds
$$\Ext_{\eJac_x}^1(M_i,M_j)=\begin{cases}
L^2, &\text{if } i=j\\
L, &\text{if } i\neq j,\end{cases}$$ giving us the tangent structure of $\Xscr_{\eJac_x}$ over $P\in\Specm(\cent(\eJac_x))$. 
\end{example}

\begin{center}
\rule{0.50\textwidth}{0.5pt}
\end{center}
\bibliographystyle{alpha}
\bibliography{ref_Arit_NC}

\end{document}